\documentclass[a4paper,11pt,reqno,noindent]{amsart}
\usepackage[centertags]{amsmath}
\usepackage{amsfonts,amssymb,amsthm} 
\usepackage[dvips]{graphicx} 
\usepackage{psfrag}
\usepackage[english]{babel}
\usepackage{newlfont}
\usepackage{color}
\usepackage[body={15cm,21.5cm},centering]{geometry} 
\usepackage{fancyhdr}
\usepackage[active]{srcltx}
\newtheorem{theo}{Theorem}
\newtheorem{lemma}{Lemma}
\newtheorem{prop}{Proposition}
\newtheorem{corollary}{Corollary}

\theoremstyle{definition}
\newtheorem{rem}{Remark}
\newtheorem{defi}{Definition}

\numberwithin{equation}{section}

\newcommand{\ho}{\mathrm{hom}}
\newcommand{\per}{\mathrm{per}}

\newcommand{\R}{\mathbb{R}}
\newcommand{\Z}{\mathbb{Z}}
\newcommand{\N}{\mathbb{N}}

\newcommand{\ee}{\mathbf{e}}

\newcommand{\mf}{\mathfrak}
\newcommand{\mb}{\mathbb}
\newcommand{\la}{\langle}
\newcommand{\ra}{\rangle}

\newcommand{\Ah}{A_\mathrm{hom}}
\newcommand{\cE}{\mathcal{E}}
\renewcommand{\o}{\omega}
\renewcommand{\P}{\mathbb{P}}
\renewcommand{\L}{\mathcal{L}}
\newcommand{\D}{\operatorname{D}}
\def\d{{\mathrm{d}}}

\mathchardef\emptyset="001F

\newcommand{\dig}[1]{\mathrm{diag}\left[ #1\right]}
\newcommand{\var}[1]{\mathrm{var}\left[#1\right]}

\newcommand{\expec}[1]{\left\langle #1 \right\rangle}
\newcommand{\step}[1]{\noindent \textit{Step} #1.}

\title{Spectral measure and approximation of homogenized coefficients}
\author[A. Gloria \& J.-C. Mourrat]{Antoine Gloria \& Jean-Christophe Mourrat}
\date{\today}
\address[Antoine Gloria]{Projet SIMPAF, INRIA Lille-Nord Europe, France}
\email{antoine.gloria@inria.fr}
\address[Jean-Christophe Mourrat]{Centre de Math\'ematiques et Informatique (CMI), Universit\'e de Provence, France}
\email{mourrat@cmi.univ-mrs.fr}
\begin{document}
\maketitle
%%----------------------------------------------------------------------------------------------------------------------

\begin{center}
\begin{minipage}{13cm}
\small{
\noindent {\bf Abstract.} 
This article deals with the numerical approximation of effective coefficients in stochastic homogenization
of discrete  linear elliptic equations.
The originality of this work is the use of a well-known \emph{abstract} spectral representation formula to design 
and analyze effective and \emph{computable} approximations of the homogenized coefficients.
In particular, we show that information on the edge of the spectrum of the generator of the environment
viewed by the particle projected on the local drift yields bounds on the approximation error, and conversely.
Combined with results by Otto and the first author in low dimension, and results by the second author in high dimension,
this allows us to prove that for any dimension $d\geq 2$, there exists an explicit numerical strategy to approximate
homogenized coefficients which converges at the rate of the central limit theorem.

\vspace{10pt}
\noindent {\bf Keywords:} 
stochastic homogenization, spectral theory, ergodic theory, numerical method.

\vspace{6pt}
\noindent {\bf 2010 Mathematics Subject Classification:} 35B27, 37A30, 65C50, 65N99.}
% id : homogenization, limit theorems

\end{minipage}
\end{center}

\bigskip

%%%%%%%%%

\section{Introduction}

\noindent We consider a discrete elliptic operator $-\nabla^*\cdot A\nabla$, where $\nabla^*\cdot$ and $\nabla$
are the discrete backward divergence and forward gradient, respectively. For all $z\in \Z^d$, $A(z)$ is the diagonal
matrix whose entries are the conductances $\omega_{z,z+\ee_i}$ of the edges $(z,z+\ee_i)$ starting at $z$, where $\{\ee_i\}_{i\in \{1,\dots,d\}}$
denotes the canonical basis of $\Z^d$.
The values of the conductances are random and their realizations are assumed to be independent and identically distributed.

\medskip
\noindent Provided that the conductances lie in a compact set of $\R^*_+$, standard homogenization results (see for instance \cite{Kunnemann-83}) ensure
that there exists some \emph{deterministic} matrix $A_\ho$ such that the solution operator of the
deterministic continuous differential operator $-\nabla \cdot A_\ho\nabla$ describes the large scale behavior of the solution operator
of the random discrete differential operator $-\nabla^*\cdot A\nabla$ almost surely.
As a by-product of this homogenization result, one obtains a characterization of the homogenized coefficients $A_\ho$: it is shown
that for every direction $\xi \in \R^d$, there exists a unique scalar field $\phi$ such that $\nabla \phi$ is stationary, 
$\expec{\nabla \phi}=0$ (vanishing expectation), which solves the corrector equation 
\begin{equation}\label{eq:intro-corr-eq}
-\nabla^*\cdot A(\xi+\nabla \phi)\,=\,0 \qquad \text{ in }\Z^d,
\end{equation}
and normalized by $\phi(0)=0$.
With this corrector, the homogenized coefficients $A_\ho$ can be characterized as
\begin{equation}\label{eq:intro-hom-coef}
\xi\cdot A_\ho \xi\,=\,\expec{(\xi+\nabla \phi)\cdot A(\xi+\nabla \phi)}.
\end{equation}

\medskip
\noindent From the practical point of view, \eqref{eq:intro-hom-coef} is not of immediate interest
since the corrector equation \eqref{eq:intro-corr-eq} has to be solved
\begin{itemize}
\item for every realization of the coefficients $\omega$,
\item on the whole $\Z^d$.
\end{itemize}
Ergodicity allows one to replace the expectation by a spatial average (on increasing domains) almost surely.
To approximate $\phi$, one usually uses $\phi_R$, the unique solution to equation \eqref{eq:intro-corr-eq} 
on some large but finite domain $Q_R=(-R/2,R/2)^d$, completed by say periodic or homogeneous Dirichlet boundary conditions.
Yet, the comparison of $\nabla \phi_R$ to $\nabla \phi$ is not obvious since $\nabla \phi_R$ and $\nabla \phi$ are not ``jointly stationary''.
In order to avoid this difficulty, Otto and the first author have used a somewhat different strategy.
We have proceeded in two steps: we first replace $\phi$ by its standard regularization $\phi_\mu$, unique stationary solution to
the modified corrector equation
\begin{equation*}
\mu\phi_\mu-\nabla^*\cdot A (\xi+\nabla \phi_\mu)\,=\,0 \qquad \text{in }\Z^d
\end{equation*}
for some small $\mu>0$.
Then, $\phi_\mu$ is replaced by $\phi_{\mu,R}$, the unique weak solution to
\begin{equation*}
\left\{
\begin{array}{rcl}
\mu\phi_{\mu,R}-\nabla^*\cdot A(\xi+\nabla \phi_{\mu,R})&=&0 \qquad \text{in }Q_R\cap \Z^d,\\
\phi_{\mu,R}&=&0 \qquad \text{on }\Z^d \setminus Q_R.
\end{array}
\right.
\end{equation*}
The advantages are twofold:
\begin{itemize}
\item $\nabla \phi$ and $\nabla \phi_\mu$ are jointly stationary, which is of great help for the analysis,
\item $\phi_\mu$ is accurately approximated by $\phi_{\mu,R}$ on domains of the form
$Q_L=(-L/2,L/2)^d$ provided that $(R-L)\sqrt{\mu}\gg 1$, due to the
exponential decay of the Green's function associated with $\mu-\nabla^* \cdot A\nabla$ in $\Z^d$
(see \cite{Gloria-10}), so that we only focus on $\phi_\mu$ and not $\phi_{\mu,R}$ from now on.
\end{itemize}

\medskip
\noindent In particular, we may approximate $A_\ho$ by the following average
\begin{equation*}
\xi\cdot A_{\mu,1,L}\xi \,:=\,\int_{Q_L}(\xi+\nabla \phi_\mu)\cdot A (\xi+\nabla \phi_\mu)\chi_L(x)\d x,
\end{equation*}
where $\chi_L$ is a smooth mask supported on $Q_L$ and of mass one.
In \cite{Gloria-Otto-09}, we have proved that the $L^2$-norm of the error in probability takes the form
\begin{equation}\label{intro:error-estim}
\expec{\big(\xi\cdot A_{\mu,1,L}\xi-\xi\cdot A_\ho\xi\big)^2} \,=\,\var{\xi\cdot A_{\mu,1,L}\xi}+\big( \xi\cdot (A_{\mu,1}-A_\ho)\xi\big)^2,
\end{equation}
where 
\begin{equation*}
\xi\cdot A_{\mu,1}\xi \,:=\,\expec{(\xi+\nabla \phi_\mu)\cdot A (\xi+\nabla \phi_\mu)}.
\end{equation*}
The first term of the r.~h.~s. of \eqref{intro:error-estim} is stochastic in nature and corresponds to the variance of the approximation of the homogenized coefficients, 
whereas the second
term is a systematic deterministic error related to the fact that we have modified the corrector equation.

\medskip
\noindent
In \cite{Gloria-Otto-09}, we have proved that the stochastic error depends on the dimension and has the scaling of the central limit theorem
(in other words the energy density of the corrector behaves as if it were independent from site to site): there exists $q$ depending only on the ellipticity constants $\alpha,\beta$ such that
\begin{equation}\label{eq:intro-13}
\var{A_{\mu,1,L}}^{1/2} \,\lesssim \,\left|
\begin{array}{rcl}
 L^{-1} \ln^q \mu &\text{if}&d=2 ,\\
 L^{-d/2} &\text{if}&d>2. 
\end{array}
\right. 
\end{equation}
The systematic error has been identified in \cite{Gloria-Otto-09b}. It also depends on the dimension for $d<5$, but saturates at $d=5$: 
there exists $q$ depending only on the ellipticity constants $\alpha,\beta$
such that
\begin{equation}
\label{eq:intro-12a}
|A_{\mu,1}-A_\ho| \,\lesssim \,\left|
\begin{array}{rcl}
\mu \ln^q \mu^{-1}&\text{if}&d=2 ,\\
\mu^{3/2}&\text{if}&d=3 , \\
\mu^2\ln \mu^{-1} &\text{if}&d=4, \\
\mu^2 &\text{if}&d>4.
\end{array}
\right. 
\end{equation}
These two estimates are optimal (up to some possible logarithmic corrections for $d=2$).
In order to use $\phi_{\mu,R}$ as a proxy for $\phi_\mu$ on $Q_L$, at first order we may take $\mu^{-1} \sim L^2 \sim R^2$.
Hence, the stochastic error dominates up to $d=8$, so that the convergence rate of the numerical strategy is optimal (it coincides with the central limit theorem scaling, which is an upper bound):
\begin{equation*}
\expec{\big(\xi\cdot A_{L^{-2},1,L}\xi-\xi\cdot A_\ho\xi\big)^2}^{1/2} \,\lesssim \,
\left|
\begin{array}{rcl}
L^{-1}\ln^qL &\mbox{if}&d=2,\\
L^{-d/2} &\mbox{if}&2< d\leq 8.
\end{array}
\right.
\end{equation*}
Yet, for $d>8$, the systematic error dominates and the numerical strategy is not optimal any longer:
\begin{equation*}
\expec{\big(\xi\cdot A_{L^{-2},1,L}\xi-\xi\cdot A_\ho\xi\big)^2}^{1/2} \,\lesssim \,
\begin{array}{rcl}
L^{-4}&\mbox{if}&8<d.
\end{array}
\end{equation*}

\bigskip
\noindent The aim of this paper is to introduce new formulas for the approximation of $A_\ho$ using the modified corrector $\phi_\mu$ (possibly with
different $\mu$'s) in order to reduce the \emph{systematic error}. 
In early and seminal papers on stochastic homogenization (for instance \cite{Papanicolaou-Varadhan-79} and \cite{Kipnis-Varadhan-86}),
spectral analysis has been used to prove uniqueness of correctors, and devise a spectral representation formula for $A_\ho$.
In particular, denoting by $-\mathcal{L}$ the generator of the environment viewed by the particle, and by $e_{\mathfrak{d}}$
its spectral measure projected on the local drift $\mathfrak{d}=\nabla^* \cdot A\xi$ (see Section~\ref{sec:sto}), we have
$$
\xi\cdot A_\ho\xi\,=\,\expec{\xi\cdot A\xi}-\int_{\R^+}\frac{1}{\lambda}\d e_{\mathfrak{d}}(\lambda).
$$
As noticed by the second author in \cite{Mourrat-10}, $A_{\mu,1}$ can also be written in terms 
of the spectral measure $e_{\mathfrak{d}}$ (see Section~\ref{sec:perio} for details):
\begin{eqnarray*}
%\xi\cdot \tilde A_{\mu,1} \xi &=&\expec{\xi\cdot A\xi}-\int_{\R^+}\frac{\lambda+\mu}{(\mu+\lambda)^2} \d e_{\mathfrak{d}}(\lambda) \\
%&=&   \xi\cdot A_\ho\xi + (\mu+\mu^2) \int_{\R^+}\frac{1}{\lambda(\mu+\lambda)^2} \d e_{\mathfrak{d}}(\lambda).
\xi\cdot A_{\mu,1} \xi &=&\expec{\xi\cdot A\xi}- \int_{\R^+}\frac{\lambda+2\mu}{(\mu+\lambda)^2} \d e_{\mathfrak{d}}(\lambda) \\
&=& \xi\cdot A_\ho\xi + \mu^2 \int_{\R^+}\frac{1}{\lambda(\mu+\lambda)^2} \d e_{\mathfrak{d}}(\lambda).
\end{eqnarray*}
%
%so that it becomes clear that the approximation $A_\mu$ is better than $\tilde A_\mu$.
The key idea of the present paper is to use this spectral representation in order to design approximations of $A_\ho$ at an abstract level first, and then go
back to physical space and obtain formulas in terms of the modified correctors $\phi_\mu$.
We shall actually introduce, for every integer $k \ge 1$, an approximation $A_{\mu,k}$ of $A_\ho$ defined in terms of $\phi_\mu,\dots,\phi_{2^{k-1}\mu}$,
and prove that, up to logarithmic corrections, the difference $|A_\ho - A_{\mu,k}|$ is bounded in our discrete stochastic setting by
\begin{equation}\label{eq:intro-14}
\left|
\begin{array}{rcl}
\mu^{\min\{2k,d/2\}}&\text{if}&d\le 6 , \\
\mu^{\min\{2k,\max(3,d/2-3)\} } &\text{if}&d > 6 ,
\end{array}
\right.
\end{equation}
(see Theorem~\ref{theo:main-sto} for a more precise statement).
The systematic error associated with the new approximations
% can thus be made of arbitrarily high order in the periodic case, and it 
can be made of a higher order than (\ref{eq:intro-12a}) as soon as $d \geq 4$. 
The proof of these estimates relies on the observation that the systematic error is controlled by the edge of the spectrum $e_{\mathfrak{d}}((0,\mu))$.
In turn, the systematic error also controls the edge of the spectrum (see Theorem~\ref{theo:sto-2} for a precise statement), so that estimating
the systematic error is equivalent to quantifying $e_{\mathfrak{d}}((0,\mu))$.

\medskip
\noindent As we shall also prove, the variance estimate \eqref{eq:intro-13} is unchanged if $A_{\mu,1}$ is replaced by $A_{\mu,k}$ for all $k \geq 1$.
In particular, if we keep $\mu^{-1} \sim L^2$, we obtain a numerical strategy whose convergence rate is optimal with respect to the central limit theorem scaling in the stochastic case, for any $d \geq 2$.
This improves and completes for $d > 8$ the series of papers \cite{Gloria-Otto-09,Gloria-Otto-09b,Gloria-10} 
by Otto and the first author on quantitative estimates in stochastic homogenization of discrete elliptic equations.
In turn, we also obtain ``optimal'' bounds on $e_{\mathfrak{d}}((0,\mu))$ up to $d=6$ (see Theorem~\ref{theo:exponents}), 
thus improving the corresponding results of the second author in \cite{Mourrat-10}.

\medskip
\noindent
Note however that the bounds \eqref{eq:intro-14} are not yet optimal: the systematic error is expected to behave as $\mu^{\min\{2k,d/2\}}$ in
any dimension (up to logarithmic corrections), see also Remark~\ref{rem:edge} for the equivalent statement in terms of the
edge of the spectrum.
We wish to address this issue in a future work.

\bigskip
\noindent
The article is organized as follows.
Although the main focus of this work is on stochastic homogenization of discrete elliptic equations, we first describe
the strategy on the elementary case of periodic homogenization of continuous elliptic equations in Section~\ref{sec:perio}
(this new strategy may indeed be valuable to numerical homogenization methods, see in particular \cite{Gloria-09} for related issues).
We introduce the spectral decomposition formula for the homogenized coefficients.
The binomial formula then provides with natural approximations of the homogenized coefficients in terms of the 
associated spectral measure. We conclude the section by rewriting these formulas in physical space using 
solutions to the modified corrector equation, which yields new \emph{computable} approximations of the homogenized coefficients.
In particular, this generalizes the method introduced in \cite{Gloria-09} and makes the systematic error decay arbitrarily fast.
Some numerical tests displayed in Appendix~\ref{append-perio} illustrate the sharpness of the analysis.

\medskip

\noindent
We turn to the core of this article in Section~\ref{sec:sto}: the stochastic homogenization of discrete elliptic equations.
We first recall the spectral decomposition of the generator of the environment viewed by the particle. 
The algebra is the same as in the continuous periodic case, so that the formulas
we obtain in Section~\ref{sec:perio} adapt mutatis mutandis to the discrete stochastic case.
Yet, the error analysis is more subtle. We show that the asymptotic behavior of the systematic error is 
driven by the behavior of the edge of the spectrum of the generator. Using results of \cite{Mourrat-10} in high dimension, and results in the
spirit of \cite{Gloria-Otto-09b} (see Lemma~\ref{lem:app} and Appendix~\ref{append})
in low dimension, we obtain estimates on the edge of this spectrum, which show that the systematic error is effectively reduced 
in high dimensions (although our bounds are not optimal when $d >6$). We then note that the variance estimates derived in \cite{Gloria-Otto-09} also hold for these
approximations, thus concluding the error analysis of the numerical strategy. 

\bigskip
\noindent We will make use of the following notation:
\begin{itemize}
\item $d\geq 2$ is the dimension;
\item In the discrete case, $\int_{\Z^d} \d x$ denotes the sum over $x\in\Z^d$, and $\int_{D}\d x$ denotes the sum over $x\in\Z^d$ such that $x\in D$, $D$ open subset of $\R^d$;
\item $\expec{\cdot}$ is the average in the periodic case, and the expectation in the stochastic case;
\item $\var{\cdot}$ is the variance in the stochastic case;
\item $\lesssim$ and $\gtrsim$ stand for $\leq$ and $\geq$ up to a multiplicative constant which only depends on the dimension $d$ and the constants $\alpha,\beta$ (the
ellipticity constants of the matrix $A$, see Definitions~\ref{defi:A-per} and \ref{defi:envi}) if not otherwise stated; 
\item when both $\lesssim$ and $\gtrsim$ hold, we simply write $\sim$;
\item we use $\gg $ instead of $\gtrsim$ when the multiplicative constant is (much) larger than $1$;
\item $(\ee_1,\dots,\ee_d)$ denotes the canonical basis of $\R^d$.
\end{itemize}

\section{The continuous periodic case}\label{sec:perio}

\begin{defi}\label{defi:A-per}
Let $A:\R^d\to \mathcal{M}_d(\R)$ be a $Q$-periodic symmetric diffusion matrix which is uniformly
continuous and coercive with constants $\beta \geq \alpha >0$: for 
almost all $x\in Q$ and all $\xi\in \R^d$, $|A\xi|\leq \beta |\xi|$ and $\xi\cdot A\xi \geq \alpha|\xi|^2$.
The associated homogenized matrix $A_\ho$ is characterized for all $\xi\in \R^d$ by
\begin{eqnarray*}
\xi\cdot A_\ho \xi &=&\expec{(\xi+\nabla \phi)\cdot A(\xi+\nabla \phi)}\\
%& = &\frac{1}{N^d}\int_{Q}(\xi+\nabla \phi)\cdot A(\xi+\nabla \phi)\d x,
& = &\int_{Q}(\xi+\nabla \phi)\cdot A(\xi+\nabla \phi)\d x,
\end{eqnarray*}
where $\expec{\cdot }$ denotes the average on the periodic cell $Q$, and $\phi$ is the unique
$Q$-periodic weak solution to 
\begin{equation}\label{eq:corr}
-\nabla \cdot A (\xi+\nabla \phi) \,=\,0
\end{equation}
with zero average $\expec{\phi}=0$.
\end{defi}

\medskip
\noindent Let us define 
%$\mathcal{E}:H^1_\per(Q)\times H^1_\per(Q)\to \R^+, (\psi,\chi)\mapsto \frac{1}{N^d}\int_{Q}\nabla \psi\cdot A\nabla \chi \d x$
$\mathcal{E}:H^1_\per(Q)\times H^1_\per(Q)\to \R^+, (\psi,\chi)\mapsto \int_{Q}\nabla \psi\cdot A\nabla \chi \d x$
the bilinear form associated with $A$. We call the quadratic form $\psi \to \cE(\psi,\psi)$ the \emph{Dirichlet form}. One may write the homogenized matrix as
\begin{equation}\label{eq:hom-Dir}
\xi\cdot A_\ho \xi \,=\,\expec{\xi\cdot A\xi}-\mathcal{E}(\phi,\phi).
\end{equation}
Indeed, the weak formulation of \eqref{eq:corr} implies 
\begin{equation}\label{eq:corr-wf}
\int_{Q}\nabla \phi\cdot A(\xi+\nabla \phi)\d x\,=\,0,
\end{equation}
and therefore
\begin{equation*}
\int_{Q}\xi \cdot A \nabla \phi \d x\,=\, \int_{Q}\nabla \phi\cdot A \xi \d x\,\stackrel{\eqref{eq:corr-wf}}{=}\,-\int_{Q}\nabla \phi\cdot A \nabla \phi \d x\,=\,-\mathcal{E}(\phi,\phi).
\end{equation*}
The objective of this section is to use a spectral decomposition to design approximations for %the Dirichlet form 
$\mathcal{E}(\phi,\phi)$.

\subsection{Spectral decomposition}

\begin{defi}\label{defi:decomp-spec-perio}
Let $A$ be as in Definition~\ref{defi:A-per}.
We let $\L^{-1}$ denote the inverse of the elliptic operator $\L=-\nabla \cdot A\nabla$ with periodic boundary conditions on $L^2_0(Q)=\{v\in L^2(Q)\,|\,\int_Q v(x)dx=0\}$.
It is a well-defined compact operator by generalized Poincar\'e's inequality, Riesz's, and Rellich's theorems.
\end{defi}
%  on $D(\L)=C^\infty_\per(Q)$ is unbounded, positive and essentially self-adjoint.
% Since its inverse is compact by Rellich' theorem, its spectrum $\{\lambda_i\}_{i\in \N}$ is purely discrete, and the associated eigenfunctions $\psi_i \in H^1_\per(Q)$
% form an orthonormal basis of $L^2(Q)$.
% % operator on $H^1_\per(Q)$.
% %We denote by $\{\lambda_i\}_{i\in \N}$ its spectrum and by $\{\psi_i\}_{i\in\N}$ the associated orthonormal 
% %basis of $H^1_\per(Q)$.
% \end{defi}
%
\noindent 
By Hilbert-Schmidt's theorem, 
there exist an orthonormal basis $\{\psi_i\}_{i>0}$ of $L^2_0(Q)$ and positive eigenvalues $\{\lambda_i\}_{i>0}$ (in increasing order)
such that for all $i>0$, $\L^{-1}\psi_i=\frac{1}{\lambda_i}\psi_i$.
By definition, $\psi_i\in H^1_\per(Q)$.
Setting $\psi_0\equiv 1$ and $\lambda_0=0$, one may then characterize $H^{1}_\per(Q)$ as
\begin{equation*}
H^1_\per(Q)\,=\,\left\{u=\sum_{i\in \N}\alpha_i \psi_i \,\bigg| \, \sum_{i\in\N}(1+\lambda_i)\alpha_i^2<\infty\right\}.
\end{equation*}
By Riesz's representation theorem, this also implies for the dual $H^{-1}_\per(Q)$ of $H^{1}_\per(Q)$:
\begin{equation}\label{eq:carac-H-1}
H^{-1}_\per(Q)\,=\,\left\{f=\sum_{i\in \N}\beta_i \psi_i \,\bigg| \, \sum_{i\in\N}\frac{\beta_i^2}{1+\lambda_i}<\infty\right\}.
\end{equation}
Hence, for all $f\in H^{-1}_\per(Q)$ such that $\expec{f,1}_{H^{-1}_\per,H^1_\per}=0$, the unique weak solution $u\in H^1_\per(Q)$ to 
\begin{equation*}
-\nabla \cdot A\nabla u\,=\,f
\end{equation*}
is given by
\begin{equation*}
u\,=\,\sum_{i\in\N\setminus \{0\}}\frac{\expec{f,\psi_i}_{H^{-1}_\per,H^1_\per}}{\lambda_i}\psi_i.
\end{equation*}
\medskip
\noindent For all $f\in H^{-1}_\per(Q)$ such that $\expec{f,1}_{H^{-1}_\per,H^1_\per}=0$,
we define the spectral measure $e_f$ of $\L$ projected on $f$ by
\begin{equation}\label{eq:spectral-measure}
e_f\,=\,\sum_{i\in \N} \expec{f,\psi_i}_{H^{-1}_\per,H^1_\per}\delta_{\lambda_i},
\end{equation}
where $\delta_{\lambda_i}$ is the Dirac mass on $\lambda_i$. 
The above characterizations of $H^1$ and $H^{-1}$ then allow us to give a mathematical meaning 
to the formal functional calculus
\begin{equation*}
\expec{f,\Psi(\L)(f)}_{H^{-1}_\per,H^1_\per}\,=\,\int_{\R^+}\Psi(\lambda)\d e_f(\lambda),
\end{equation*}
for every continuous function $\Psi:[0,+\infty)\to \R$ such that
$\lambda \Psi(\lambda) \,\lesssim \, 1$ as $\lambda\to \infty$.

\medskip
\noindent We are now in position to express the Dirichlet form of the corrector $\phi$
in terms of the spectral measure projected on the ``local drift'' 
\begin{equation}\label{eq:def-local-drift-per}
\mf{d}:=\nabla \cdot A\xi \in H^{-1}_\per(Q).
\end{equation}
In particular,
\begin{eqnarray}
\mathcal{E}(\phi,\phi)&=&\expec{\L\phi,\phi}_{H^{-1}_\per,H^1_\per} \nonumber \\
&=& \expec{\L\L^{-1}\mf{d},\L^{-1}\mf{d}}_{H^{-1}_\per,H^1_\per} \nonumber\\
&=&\int_{\R^+}\frac{1}{\lambda}\d e_{\mf{d}}(\lambda).\label{eq:dir-spec-1}
\end{eqnarray}

\medskip
\noindent Let us then turn to the approximation of $A_\ho$ used in \cite{Gloria-09}, that is
\begin{equation}
\label{eq:Ahom1}
\xi\cdot A_\mu\xi\,:=\,\expec{(\xi+\nabla \phi_\mu)\cdot A(\xi+\nabla \phi_\mu)},
\end{equation}
where $\phi_\mu\in H^1_\per(Q)$ is the unique weak solution to the modified corrector equation
\begin{equation}\label{eq:corr-mu}
\mu\phi_\mu-\nabla \cdot A (\xi+\nabla \phi_\mu) \,=\,0,
\end{equation}
that we more compactly write as $(\mu+\L)\phi_\mu\,=\,\mf{d}$.
In this case, the weak formulation of the equation implies
\begin{equation}
\label{eq:Ahom1weak}
\int_{Q}\nabla \phi_\mu\cdot A(\xi+\nabla \phi_\mu)\d x\,=\,-\mu \int_{Q}\phi_\mu^2\d x,
\end{equation}
so that the defining formula for $A_\mu$ turns into 
\begin{equation}\label{eq:hom-Dir-mu}
\xi\cdot A_\mu \xi \,=\,\expec{\xi\cdot A\xi}-\mathcal{E}(\phi_\mu,\phi_\mu)-2\mu \expec{\phi_\mu^2}.
\end{equation}
Proceeding as above, we rewrite the last two terms of the r.~h.~s. of \eqref{eq:hom-Dir-mu} as
\begin{eqnarray}
\mathcal{E}(\phi_\mu,\phi_\mu)&=&\expec{\L(\mu+\L)^{-1}\mf{d},(\mu+\L)^{-1}\mf{d}}_{H^{-1}_\per,H^1_\per}  \nonumber \\
&=&\expec{\mf{d},\L (\mu+\L)^{-2}\mf{d}}_{H^{-1}_\per,H^1_\per}  \nonumber \\
&=&\int_{\R^+}\frac{\lambda}{(\mu+\lambda)^2}\d e_{\mf{d}}(\lambda),\label{eq:dir-spec-2}
\end{eqnarray}
and
\begin{eqnarray}
\expec{\phi_\mu^2}&=&\expec{(\mu+\L)^{-1}\mf{d},(\mu+\L)^{-1}\mf{d}}_{L^2,L^2} \nonumber \\
&=&\expec{\mf{d},(\mu+\L)^{-2}\mf{d}}_{H^{-1}_\per,H^1_\per}  \nonumber \\
&=&\int_{\R^+}\frac{1}{(\mu+\lambda)^2}\d e_{\mf{d}}(\lambda). \label{eq:dir-spec-3}
\end{eqnarray}

\medskip
\noindent The combination of \eqref{eq:hom-Dir}, \eqref{eq:dir-spec-1}, \eqref{eq:hom-Dir-mu}, \eqref{eq:dir-spec-2} \& \eqref{eq:dir-spec-3} allows us to
express the difference between $A_\mu$ and $A_\ho$ in terms of the spectral measure of $\L$
projected on the local drift $\mf{d}$, as observed in \cite[Addendum]{Mourrat-10}:
\begin{eqnarray*}
\xi\cdot (A_\mu-A_\ho)\xi &=&\mathcal{E}(\phi,\phi)-\mathcal{E}(\phi_\mu,\phi_\mu)-2\mu \expec{\phi_\mu^2} \\
&=&\int_{\R^+}\left( \frac{1}{\lambda}- \frac{\lambda}{(\mu+\lambda)^2}-\frac{2\mu}{(\mu+\lambda)^2}\right) \d e_{\mf{d}}(\lambda) \\
&=&\int_{\R^+}\frac{\mu^2}{\lambda (\mu+\lambda)^2}\d e_{\mf{d}}(\lambda) .
\end{eqnarray*}
Not only does this identity suggest that $|A_\mu-A_\ho| \sim \mu^2$ (as proved by a different approach in \cite{Gloria-09}),
but it also gives a strategy to construct approximations of $A_\ho$ at any order.
In particular, for all $k\in \N$, we write
\begin{eqnarray*}
\xi\cdot A_\ho\xi &=& \expec{\xi\cdot A \xi}-\int_{\R^+}\frac{(\mu+\lambda)^{2k}}{\lambda(\mu+\lambda)^{2k}}\d e_{\mf{d}}(\lambda)\\
&=&\expec{\xi\cdot A \xi}-\int_{\R^+}\frac{(\mu+\lambda)^{2k}-\mu^{2k}}{\lambda(\mu+\lambda)^{2k}} \d e_{\mf{d}}(\lambda)-\int_{\R^+}\frac{\mu^{2k}}{\lambda(\mu+\lambda)^{2k}} \d e_{\mf{d}}(\lambda)
\end{eqnarray*}
and set 
\begin{eqnarray}
\xi\cdot \tilde A_{\mu,k}\xi&:=&\expec{\xi\cdot A\xi}-\int_{\R^+}\frac{(\mu+\lambda)^{2k}-\mu^{2k}}{\lambda(\mu+\lambda)^{2k}} \d e_{\mf{d}}(\lambda) \nonumber \\
&=& \expec{\xi\cdot A\xi}-\sum_{j=1}^{2k-1} \binom{2k}{j} \int_{\R^+}\frac{\mu^j \lambda^{2k-1-j} }{(\mu+\lambda)^{2k}} \d e_{\mf{d}}(\lambda) \label{eq:approx-mu-k}.
\end{eqnarray}
Note that the only operator which has to be inverted to compute $\tilde A_{\mu,k}$ is indeed $\mu+\L$, and not $\L$, as desired.
In addition, this definition implies
\begin{eqnarray*}
\xi\cdot (\tilde A_{\mu,k}-A_\ho)\xi &=&\int_{\R^+}\frac{\mu^{2k}}{\lambda(\mu+\lambda)^{2k}} \d e_{\mf{d}}(\lambda),
\end{eqnarray*}
which suggests that the error is now of order $\mu^{2k}$.

\medskip
\noindent In view of formula \eqref{eq:approx-mu-k}, the effective computation of $(\mu+\L)^{-k}\mf{d}$ is needed in practice to obtain $\tilde A_{\mu,k}$.
This is a big handicap for the numerical method since the numerical inversion of $\mu+\L$ has to be iterated $k$ times, 
which dramatically magnifies the numerical error.
Fortunately, one may use a sligthly different approximation of $A_\ho$ which avoids this drawback, as shown in the following subsection.

\subsection{Abstract approximations}

Let us first introduce functions $\mf{d}_{\mu,k}$, which are defined as linear combinations of $\phi_{\mu},\dots,\phi_{2^{k-1}\mu}$ (and therefore easily computable) and will serve as
substitutes for $(\mu+\L)^{-k}\mf{d}$.
\begin{defi}\label{def:mfd-mu-k}
Let $A$ and $Q$, $\L$, and $\mf{d}$ be as in Definition~\ref{defi:A-per}, Definition~\ref{defi:decomp-spec-perio}, and \eqref{eq:def-local-drift-per}, respectively.
For all $\mu>0$, the sequence of functions $\mf{d}_{\mu,k}\in H^1_\per(Q)$ is defined by its first term
\begin{equation}
\mf{d}_{\mu,1}=\phi_\mu=(\mu+\L)^{-1}\mf{d}, \quad c_1=1,
\end{equation}
and by the induction rule
\begin{equation}
\mf{d}_{\mu,k+1}=c_k\mu^{-1}(\mf{d}_{\mu,k}-\mf{d}_{2\mu,k}), \quad c_{k+1}= \left(\frac{2}{c_k} + 1\right)^{-1}. \label{eq:def-mfd-rec}
\end{equation}
\end{defi}
\noindent Defined this way, the functions $\mf{d}_{\mu,k}$ satisfy the following fundamental properties:
\begin{prop}\label{prop:mfd-rec}
Let $\mf{d}_{\mu,k}$ be as in Definition~\ref{def:mfd-mu-k}, then for all $\mu>0$ and $k\geq 1$, we have
\begin{eqnarray}
\mf{d}_{\mu,k+1} &=&(\mu+\L)^{-1} \mf{d}_{2\mu,k},\label{eq:Lmfd-muL}\\
\L\mf{d}_{\mu,k+1}&=&(1+c_k) \mf{d}_{2\mu,k}-c_k \mf{d}_{\mu,k}. \label{eq:Lmfd}
\end{eqnarray}
\end{prop}
\begin{proof}
Identity~\eqref{eq:Lmfd} is a direct consequence of \eqref{eq:def-mfd-rec} \& \eqref{eq:Lmfd-muL}, and we
only need to prove the latter. We proceed by induction. Let us first check that it is indeed true for $k = 1$. By definition of $\mf{d}_{\mu,1}$, we have
\begin{equation}
\label{mfdmu111}
(\mu+\L) \mf{d}_{\mu,1} = \mf{d},
\end{equation}
and as a consequence,
\begin{equation}
\label{mfdmu112}
(\mu+\L) \mf{d}_{2\mu,1} = \mf{d} - \mu \ \mf{d}_{2\mu,1}.
\end{equation}
Combining (\ref{mfdmu111}) and (\ref{mfdmu112}), one obtains~:
$$
(\mu+\L) (\mf{d}_{\mu,1} - \mf{d}_{2\mu,1}) = \mu \ \mf{d}_{2\mu,1},
$$
from which it follows that $\mf{d}_{\mu,2} = (\mu+\L)^{-1} \mf{d}_{2\mu,1}$. 
Let us now assume that (\ref{eq:Lmfd-muL}) is satisfied at level $k \ge 1$. Similarly, we have
\begin{equation*}
\label{mfdmu11}
(\mu+\L) \mf{d}_{\mu,k+1} = \mf{d}_{2\mu,k},
\end{equation*}
\begin{equation*}
\label{mfdmu12}
(\mu+\L) \mf{d}_{2\mu,k+1} = \mf{d}_{4\mu,k} - \mu \ \mf{d}_{2\mu,k+1}.
\end{equation*}
Using these equalities, together with the definition (\ref{eq:def-mfd-rec}) of $\mf{d}_{\mu,k+1}$, we are led to
$$
(\mu+\L) (\mf{d}_{\mu,k+1} - \mf{d}_{2\mu,k+1}) = \mf{d}_{2\mu,k} - \mf{d}_{4\mu,k} + \mu \ \mf{d}_{2\mu,k+1} = \mu\left(\frac{2}{c_k} + 1\right) \mf{d}_{2\mu,k+1},
$$
and thus, $\mf{d}_{\mu,k+2} = (\mu+\L)^{-1} \mf{d}_{2\mu,k+1}$.
\end{proof}
\noindent
In order to be consistent with (\ref{eq:Lmfd-muL}), we set $\mf{d}_{\mu,0} = \mf{d}$ for all $\mu>0$.
%By definition,  
%
%\begin{equation*}
%\mf{d}_{\mu,1}=(\mu+\L)^{-1}\mf{d}=(\mu+\L)^{-1}\mf{d}_{2\mu,0},
%\end{equation*}
%
%provided we set $\mf{d}_{\mu,0}:=\mf{d}$ for all $\mu>0$.
%Hence \eqref{eq:Lmfd-muL} holds for $k=0$.
%Let us assume that \eqref{eq:Lmfd-muL} holds for some $k\in \N$. In particular, this implies
%
%\begin{eqnarray*}
%(\mu+\L) \mf{d}_{\mu,k+1} &=& \mf{d}_{2\mu,k},\\
%(\mu+\L) \mf{d}_{2\mu,k+1} &=& \mf{d}_{4\mu,k} - \mu \ \mf{d}_{2\mu,k+1}.
%\end{eqnarray*}
%
%Using in addition the defining equation \eqref{eq:def-mfd-rec}, this yields
%
%\begin{equation*}
%(\mu+\L) (\mf{d}_{\mu,k+1}-\mf{d}_{2\mu,k+1}) \,=\, \mf{d}_{2\mu,k}-\mf{d}_{4\mu,k} +\mu \ \mf{d}_{2\mu,k+1} 
%\,=\,\mu \left(\frac{2}{c_k} + 1\right) \mf{d}_{\mu,k+1},
%\end{equation*}
%
%so that $\mf{d}_{\mu,k+2} = (\mu+\L)^{-1} \mf{d}_{2\mu,k+1}$, as desired.
%\end{proof}
%

\medskip
\noindent We are now in position to define a suitable approximation of $A_\ho$.
The idea is to use the identity 
\begin{equation*}
\frac{1}{\lambda}\,=\,\frac{(\mu+\lambda)^2(2\mu+\lambda)^2\dots(2^{k-1}\mu+\lambda)^2}{\lambda(\mu+\lambda)^2(2\mu+\lambda)^2\dots(2^{k-1}\mu+\lambda)^2}
\end{equation*}
in \eqref{eq:dir-spec-1}, expand, and take advantage of Proposition~\ref{prop:mfd-rec} to efficiently compute terms of the form
$\mf{d}_{\mu,k}=(\mu+\L)^{-1}(2\mu+\L) ^{-1}\dots(2^{k-1}\mu+\L)^{-1}\mf{d}$.
This gives rise to the following (abstract) approximations of $A_\ho$, and systematic errors:
\begin{theo}\label{th:main-perio}
Let $A$ and $A_\ho$ be the $Q$-periodic diffusion matrix and the associated homogenized diffusion matrix of Definition~\ref{defi:A-per}.
For any fixed $\xi\in \R^d$ such that $|\xi|=1$, we denote by $e_{\mf{d}}$ the spectral measure \eqref{eq:spectral-measure} of $\L=-\nabla \cdot A\nabla$ projected 
on the local drift $\mf{d}=\nabla \cdot A\xi$.
For all $k\in \N$, we let $P_k:\R\times \R\to \R$ be the polynomial given by
\begin{equation}\label{eq:def-Pk}
P_k(\mu,\lambda) \,=\, \lambda^{-1}\left((\mu+\lambda)^2(2 \mu+\lambda)^2 \cdots (2^{k-1}\mu+\lambda)^2 - 2^{k(k-1)} \mu^{2k}\right),
\end{equation}
and for all $\mu>0$, we define the approximation $A_{\mu,k}$ of $A_\ho$ by
\begin{equation}\label{eq:th1-def-app-A}
\xi\cdot A_{\mu,k}\xi\,=\,\expec{\xi\cdot A\xi}-\int_{\R^+}\frac{P_k(\mu,\lambda)}{(\mu+\lambda)^2(2 \mu+\lambda)^2 \cdots (2^{k-1}\mu+\lambda)^2 } \d e_{\mf{d}}(\lambda).
\end{equation}
Then the systematic error satisfies
\begin{equation}
%0\,\leq\, \xi\cdot (A_{\mu,k}-A_\ho)\xi \,\leq\, 2^{k(k-1)} \expec{|A|^2}\left(\frac{N^2}{4 \alpha\pi^2}\right)^{2k+1} \ \mu^{2k}.
0\,\leq\, \xi\cdot (A_{\mu,k}-A_\ho)\xi \,\leq\, 2^{k(k-1)} \expec{|A|^2}\left(\frac{1}{4 \alpha\pi^2}\right)^{2k} \ \mu^{2k}.
\end{equation}
\end{theo}
\begin{proof}
Starting point is the identity 
\begin{eqnarray*}
\xi \cdot (A_{\mu,k} - A_\ho) \xi 
& = & \int_{\R^+}\left( \frac{1}{\lambda} - \frac{P_k(\mu,\lambda)}{(\lambda+\mu)^2 \cdots (\lambda + 2^{k-1} \mu)^2} \right) \ d e_{\mf{d}}(\lambda) \\
& = & \int_{\R^+}\frac{2^{k(k-1)} \mu^{2k}}{\lambda (\lambda+\mu)^2 \cdots (\lambda + 2^{k-1} \mu)^2} \ d e_{\mf{d}}(\lambda) ,
\end{eqnarray*}
which is a direct consequence of \eqref{eq:hom-Dir}, \eqref{eq:dir-spec-1}, and \eqref{eq:th1-def-app-A}.
From this identity, and using Definition~\ref{defi:decomp-spec-perio}, we infer that the systematic error is smaller than and asymptotically equivalent 
to $C \mu^{2k}$ (as $\mu$ tends to $0$), where $C > 0$ is given by
\begin{equation}
\label{eq:majo-C}
C\,:=\,2^{k(k-1)} \int_{\R^+}\frac{1}{\lambda^{2k+1}} \ d e_{\mf{d}}(\lambda) = 2^{k(k-1)} \sum_{i \in\N} \frac{\expec{\mf{d},\psi_i}_{H^{-1}_\per,H^1_\per}^2}{\lambda_i^{2k+1}} .
\end{equation}
In order to estimate $C$ via \eqref{eq:majo-C}, we compare the spectral gap $\lambda_1$ of $\L$ to the spectral gap $\lambda_1^0$ of $- \triangle$ on $H^1_\per(Q)$.
By comparison of the two Dirichlet forms, we have
\begin{equation*}
\lambda_1 \,\geq \, \alpha \lambda_1^0.
\end{equation*}
The spectrum of the Laplace operator on $H^1_\per(Q)$ is explicitly known, and the spectral gap given by
%
%\begin{equation*}
%$\lambda^0_1\,=\,\frac{4\pi^2}{N^2}$.
$\lambda^0_1\,=\,4\pi^2$.
%\left\{\frac{4\pi^2}{N^2}(n_1^2 + \cdots + n_d^2),n\in \N^d\right\}.
%\end{equation*}^
%
Hence, recalling that $\expec{\mf{d},1}_{H^{-1}_\per,H^1_\per}=0$ and using the characterization \eqref{eq:carac-H-1}
of $H^{-1}_\per(Q)$, one may bound the r.~h.~s. of \eqref{eq:majo-C} by
\begin{eqnarray*}
C&\leq& 2^{k(k-1)} \frac{1}{(\alpha \lambda^0_1)^{2k}} \sum_{i \in \N}\frac{\expec{\mf{d},\psi_i}_{H^{-1}_\per,H^1_\per}^2}{\lambda_i} \\
%&= &  2^{k(k-1)} \|\mf{d}\|_{H^{-1}_\per}^2\left(\frac{N^2}{4\alpha\pi^2}\right)^{2k+1} \\
&= &  2^{k(k-1)} \|\mf{d}\|_{H^{-1}_\per}^2\left(\frac{1}{4\alpha\pi^2}\right)^{2k} \\
%&\leq &  2^{k(k-1)} \expec{|A|^2}\left(\frac{N^2}{4 \alpha\pi^2}\right)^{2k+1} ,
&\leq &  2^{k(k-1)} \expec{|A|^2}\left(\frac{1}{4 \alpha\pi^2}\right)^{2k} ,
\end{eqnarray*}
as desired.
\end{proof}

\subsection{New formulas for the approximation of homogenized coefficients} \label{sec:new_fo_per}

In this subsection, we show how to rewrite the approximations $A_{\mu,k}$ of $A_\ho$ introduced in Theorem~\ref{th:main-perio}
in terms of the modified correctors $\phi_\mu,\phi_{2\mu},\dots,\phi_{2^{k-1}\mu}$.
We proceed by induction.
\begin{prop}\label{prop:new-fo-PDE}
Let $c_k$ be as in Definition~\ref{def:mfd-mu-k}.
We define the sequence $\{a_{k,i}\}_{k\geq 1,i\in \{0,\dots,k-1\}}$ by $a_{1,0}=1$ and the induction rules
\begin{eqnarray*}
a_{k+1,0}&=&c_k a_{k,0} ,\\
a_{k+1,i}&=&c_ka_{k,i}-2^{1-k}c_ka_{k,i-1} \qquad \text{ for }i\in \{1,k-1\} ,\\
a_{k+1,k}&=&-2^{1-k}c_ka_{k,k-1}.
\end{eqnarray*}
Within the assumptions and notation of Theorem~\ref{th:main-perio}, the approximations $A_{\mu,k}$ of $A_\ho$
satisfy the formula: for all $\xi\in \R^d$,
\begin{eqnarray}\label{eq:new-fo-PDE}
\xi\cdot A_{\mu,k}\xi &=&
%\sum_{i=0}^{k-1}\delta_{k,i}\expec{(\xi+\nabla \phi_{2^i\mu})\cdot A (\xi+\nabla \phi_{2^i\mu})} 
\expec{(\xi+\nabla \phi_{\mu})\cdot A (\xi+\nabla \phi_{\mu})} 
\nonumber \\
&&\qquad \qquad +\mu\sum_{i=0}^{k-1}\eta_{k,i}\expec{\phi_{2^i\mu}^2}+\mu\sum_{i=0}^{k-1}\sum_{j>i}^{k-1}\nu_{k,i,j}\expec{\phi_{2^i\mu}\phi_{2^j\mu}},
\end{eqnarray}
where the $\{\phi_{2^i\mu}\}_{i\in \N}$ are the modified correctors associated with $\xi$ through \eqref{eq:corr-mu},
and the coefficients 
%$\{\delta_{k,i}\}_{k\geq 1,0\leq i<k}$, 
$\{\eta_{k,i}\}_{k\geq 1,0\leq i<k}$, and $\{\nu_{k,i,j}\}_{k\geq 2,0\leq j<k,0\leq i<j}$ are defined by the initial value
%$\delta_{1,0}=1,\eta_{1,0}=0$, 
$\eta_{1,0}=0$, and the induction rules
\begin{eqnarray*}
\eta_{k+1,i}&=&\eta_{k,i}+(2^{k(k-1)+i}-2^{k^2+1})a_{k+1,i}^2\qquad \text{ for }i\in \{0,k-1\},\\
\eta_{k+1,k}&=&-2^{k^2}a_{k+1,k}^2,\\
\nu_{k+1,i,k}&=&\big(2^{k(k-1)+i}-3\times 2^{k^2}\big)a_{k+1,i}a_{k+1,k},\\
\nu_{k+1,i,j}&=&\nu_{k,i,j}+\big(2^{k(k-1)}(2^i+2^j)-2^{k^2+2}\big)a_{k+1,i}a_{k+1,j}\qquad \text{ for }j\in \{0,k-1\}.
\end{eqnarray*}
Note that $\{\nu_{k,i,j}\}_{k\geq 1,0\leq j<k,0\leq i<j}$ does not require further initialization.
\end{prop}
\begin{proof}
We proceed in four steps.

\medskip
\step{1} Proof that for all $k\geq 1$,
\begin{equation}\label{eq:pr-prop2-1.1}
\xi\cdot A_{\mu,k+1}\xi \,=\,\xi\cdot A_{\mu,k}\xi-2^{k(k-1)}\mu^{2k}\mathcal{E}(\mf{d}_{\mu,k+1},\mf{d}_{\mu,k+1})-2^{k^2+1}\mu^{2k+1}\expec{\mf{d}_{\mu,k+1}^2}.
\end{equation}
In order to prove \eqref{eq:pr-prop2-1.1}, we first note that the polynomials $P_{k}$ defined in \eqref{eq:def-Pk} satisfy the identity 
\begin{equation}\label{eq:pr-prop2-1.2}
P_{k+1}(\mu,\lambda)\,=\,(2^k\mu+\lambda)^2P_k(\mu,\lambda)+2^{k(k-1)}\lambda \mu^{2k}+2^{k^2+1}\mu^{2k+1}.
\end{equation}
Hence, formula~\eqref{eq:th1-def-app-A} implies
\begin{eqnarray*}
\xi\cdot A_{\mu,k+1}\xi&=&\expec{\xi\cdot A\xi}-\int_{\R^+}\frac{P_{k+1}(\mu,\lambda)}{(\mu+\lambda)^2(2 \mu+\lambda)^2 \cdots (2^{k}\mu+\lambda)^2 } \d e_{\mf{d}}(\lambda) \\
&\stackrel{\eqref{eq:pr-prop2-1.2}}{=}&\expec{\xi\cdot A\xi}-\int_{\R^+}\frac{(2^k\mu+\lambda)^2P_{k}(\mu,\lambda)}{(\mu+\lambda)^2(2 \mu+\lambda)^2 \cdots (2^{k}\mu+\lambda)^2 } \d e_{\mf{d}}(\lambda)  \\
&& -\int_{\R^+}\frac{2^{k(k-1)}\lambda \mu^{2k}+2^{k^2+1}\mu^{2k+1}}{(\mu+\lambda)^2(2 \mu+\lambda)^2 \cdots (2^{k}\mu+\lambda)^2 } \d e_{\mf{d}}(\lambda) \\
&\stackrel{\eqref{eq:th1-def-app-A}}{=}&\xi\cdot A_{\mu,k}\xi-\int_{\R^+}\frac{2^{k(k-1)}\lambda \mu^{2k}+2^{k^2+1}\mu^{2k+1}}{(\mu+\lambda)^2(2 \mu+\lambda)^2 \cdots (2^{k}\mu+\lambda)^2 } \d e_{\mf{d}}(\lambda).
\end{eqnarray*}
From \eqref{eq:Lmfd-muL} in Proposition~\ref{prop:mfd-rec}, we infer that $(\mu+\L)^{-1}\cdots(2^{k}\mu+\L)^{-1}\mf{d}=\mf{d}_{\mu,k+1}$, so that the above identity turns into
\begin{eqnarray*}
\xi\cdot A_{\mu,k+1}\xi&=&\xi\cdot A_{\mu,k}\xi-2^{k(k-1)}\mu^{2k}\expec{\L \mf{d}_{\mu,k+1},\mf{d}_{\mu,k+1}}_{H^{-1}_\per,H^1_\per}\\
&&\qquad \qquad \qquad \qquad -2^{k^2+1}\mu^{2k+1}\expec{\mf{d}_{\mu,k+1},\mf{d}_{\mu,k+1}}_{H^{-1}_\per,H^1_\per} \\
&=&\xi\cdot A_{\mu,k}\xi-2^{k(k-1)}\mu^{2k}\mathcal{E}(\mf{d}_{\mu,k+1},\mf{d}_{\mu,k+1}) -2^{k^2+1}\mu^{2k+1}\expec{\mf{d}_{\mu,k+1}^2}, 
\end{eqnarray*}
as desired.

\medskip
\step{2} Proof that for all $k\geq 1$,
\begin{equation}\label{eq:pr-prop2-2.1}
\mu^{k-1} \mf{d}_{\mu,k}\,=\,\sum_{i=0}^{k-1}a_{k,i}\phi_{2^i\mu}.
\end{equation}
We proceed by induction, and assume that \eqref{eq:pr-prop2-2.1} holds at step $k$.
The induction rule \eqref{eq:def-mfd-rec} then yields at step $k+1$
\begin{eqnarray*}
\mu^k \mf{d}_{\mu,k+1}&=&\mu^{k-1}c_k(\mf{d}_{\mu,k}-\mf{d}_{2\mu,k}) \\
&=&c_k \mu^{k-1}\mf{d}_{\mu,k}-c_k 2^{1-k}(2\mu)^{k-1}\mf{d}_{2\mu,k}\\
&=&c_k\sum_{i=0}^{k-1}a_{k,i}\phi_{2^i\mu}-c_k 2^{1-k}\sum_{i=0}^{k-1}a_{k,i}\phi_{2^{i+1}\mu} \\
&=&c_k a_{k,0}\phi_\mu+\bigg(\sum_{i=1}^{k-1}c_k(a_{k,i}-2^{1-k}a_{k,i-1})\phi_{2^i\mu}\bigg)-c_k2^{1-k}a_{k,k-1}\phi_{2^k\mu},
\end{eqnarray*}
so that $\mu^k \mf{d}_{\mu,k+1}\,=\,\sum_{i=0}^{k}a_{k,i}\phi_{2^i\mu}$, as desired.
It remains to recall that $\mf{d}_{\mu,1}\,=\,\phi_\mu$ to conclude the proof of \eqref{eq:pr-prop2-2.1}. 

\medskip
\noindent Note that $a_{2,0}=1$ and $a_{2,1}=-1$.
In particular, $a_{2,0}+a_{2,1}=0$ and the property
\begin{equation}\label{eq:pr-prop2-2.2}
\sum_{i=0}^{k-1}a_{k,i}\,=\,0
\end{equation}
follows by induction, for all $k\geq 2$.

\medskip
\step{3} Proof that for all $i,j\geq 1$,
\begin{eqnarray}
\mathcal{E}(\phi_{2^i\mu},\phi_{2^j\mu})&=&\expec{\xi\cdot A\xi}-\frac{1}{2}\bigg( \expec{(\xi+\nabla \phi_{2^i\mu})\cdot A (\xi+\nabla \phi_{2^i\mu})}  \nonumber \\
&& \qquad \qquad \qquad +\expec{(\xi+\nabla \phi_{2^j\mu})\cdot A (\xi+\nabla \phi_{2^j\mu})} \bigg) \nonumber \\
&&-\mu 2^{i-1}\expec{\phi_{2^i\mu}(\phi_{2^i\mu}+\phi_{2^j\mu})} -\mu 2^{j-1}\expec{\phi_{2^j\mu}(\phi_{2^i\mu}+\phi_{2^j\mu})} . \label{eq:pr-prop2-3.1}
\end{eqnarray}
We can easily see that (\ref{eq:pr-prop2-3.1}) holds when $i=j$. From (\ref{eq:Ahom1}) and (\ref{eq:hom-Dir-mu}), we have indeed that
%
%
%For $i=j$, \eqref{eq:pr-prop2-3.1} coincides with \eqref{eq:hom-Dir-mu}.
%We quickly recall the argument:
%
%\begin{eqnarray*}
%\mathcal{E}(\phi_{2^i\mu},\phi_{2^i\mu})&=&\expec{\nabla \phi_{2^i\mu}\cdot A \nabla \phi_{2^i\mu}}\\
%&=&\expec{(\xi+\nabla \phi_{2^i\mu})\cdot A (\xi+\nabla \phi_{2^i\mu})}-\expec{\xi\cdot A\xi}-2\expec{\nabla \phi_{2^i\mu}\cdot A\xi}.
%\end{eqnarray*}
%
%We appeal to the weak form of the defining equation \eqref{eq:corr-mu} for $\phi_{2^i\mu}$, that is
%
%\begin{equation*}
%2^i\mu \expec{\phi_{2^i\mu}^2} + \expec{\nabla \phi_{2^i\mu}\cdot A \nabla \phi_{2^i\mu}} \,=\,-\expec{\nabla \phi_{2^i\mu}\cdot A\xi},
%\end{equation*}
%
%to turn the above identity into
%
\begin{eqnarray} \label{eq:pr-prop2-3.2}
\mathcal{E}(\phi_{2^i\mu},\phi_{2^i\mu})&=&\expec{\xi\cdot A\xi}-\expec{(\xi+\nabla \phi_{2^i\mu})\cdot A (\xi+\nabla \phi_{2^i\mu})}
-2^{i+1}\mu \expec{ \phi_{2^i\mu}^2},
\end{eqnarray}
as desired.
For general $i,j\in \N$, we have, using \eqref{eq:Ahom1weak} first for $\phi_{2^j\mu}$ and then for $\phi_{2^i\mu}$,
\begin{eqnarray*}
\mathcal{E}(\phi_{2^i\mu},\phi_{2^j\mu})&=&\expec{\nabla \phi_{2^i\mu}\cdot A \nabla \phi_{2^j\mu}}\\
&=&\expec{\nabla \phi_{2^i\mu}\cdot A (\xi+\nabla \phi_{2^j\mu})}-\expec{\nabla \phi_{2^i\mu} \cdot A\xi} \\
&\stackrel{\eqref{eq:Ahom1weak}}{=}&-2^j\mu \expec{\phi_{2^i\mu}\phi_{2^j\mu}}-\expec{\nabla \phi_{2^i\mu}\cdot A (\xi+\nabla \phi_{2^i\mu})}+\expec{\nabla \phi_{2^i\mu}\cdot A \nabla \phi_{2^i\mu}} \\
&\stackrel{\eqref{eq:Ahom1weak}}{=}&-2^j\mu \expec{\phi_{2^i\mu}\phi_{2^j\mu}}+2^i\mu \expec{\phi_{2^i\mu}^2}+\mathcal{E}(\phi_{2^i\mu},\phi_{2^i\mu})\\
&\stackrel{\eqref{eq:pr-prop2-3.2}}{=}& -2^j\mu \expec{\phi_{2^i\mu}\phi_{2^j\mu}}-2^{i}\mu \expec{ \phi_{2^i\mu}^2}\\
&&\qquad \qquad +\expec{\xi\cdot A\xi}-\expec{(\xi+\nabla \phi_{2^i\mu})\cdot A (\xi+\nabla \phi_{2^i\mu})}.
\end{eqnarray*}
We conclude the proof of \eqref{eq:pr-prop2-3.1} by changing the roles of $i$ and $j$.

\medskip
\step{4} Proof of \eqref{eq:new-fo-PDE}.

\noindent In view of \eqref{eq:pr-prop2-1.1}, we have to estimate two terms. We begin with the Dirichlet form:
inserting \eqref{eq:pr-prop2-2.1} in the integral yields
\begin{equation*}
\mu^{2k}\mathcal{E}(\mf{d}_{\mu,k+1},\mf{d}_{\mu,k+1})\,=\,\sum_{i=0}^k\sum_{j=0}^ka_{k+1,i}a_{k+1,j}\mathcal{E}(\phi_{2^i\mu},\phi_{2^j\mu}).
\end{equation*}
We then appeal to \eqref{eq:pr-prop2-3.1} to turn this identity into
\begin{eqnarray*}
\lefteqn{\mu^{2k}\mathcal{E}(\mf{d}_{\mu,k+1},\mf{d}_{\mu,k+1})}\\
&=&\bigg(\sum_{i=0}^ka_{k+1,i}\bigg)^2\expec{\xi\cdot A\xi}
-\sum_{i=0}^k\sum_{j>i}^k\mu(2^i+2^j)a_{k+1,i}a_{k+1,j}\expec{\phi_{2^i\mu}\phi_{2^j\mu}}\\
&&  \qquad \qquad \qquad -\sum_{i=0}^k\mu 2^i a_{k+1,i}^2\expec{\phi_{2^i\mu}^2}\\
&&-\sum_{i=0}^ka_{k+1,i}\bigg(\sum_{j=0}^ka_{k+1,j}\bigg)\bigg(2^i\mu \expec{\phi_{2^i\mu}^2}
+\expec{(\xi+\nabla \phi_{2^i\mu})\cdot A (\xi+\nabla \phi_{2^i\mu})} \bigg) .
\end{eqnarray*}
Taking into account \eqref{eq:pr-prop2-2.2}, we finally have
\begin{multline}\label{eq:pr-prop2-4.1}
\mu^{2k}\mathcal{E}(\mf{d}_{\mu,k+1},\mf{d}_{\mu,k+1})
\\
\,=\,-\sum_{i=0}^k\sum_{j>i}^k\mu(2^i+2^j)a_{k+1,i}a_{k+1,j}\expec{\phi_{2^i\mu}\phi_{2^j\mu}}-\sum_{i=0}^k\mu 2^i a_{k+1,i}^2\expec{\phi_{2^i\mu}^2}.
\end{multline}
We now turn to the last term of the r.~h.~s. of \eqref{eq:pr-prop2-1.1} and appeal to \eqref{eq:pr-prop2-2.1}:
\begin{eqnarray}
\mu^{2k+1}\expec{\mf{d}_{\mu,k+1}^2}&=&\sum_{i=0}^k \mu a_{k+1,i}^2\expec{\phi_{2^i\mu}^2} 
+2\sum_{i=0}^k\sum_{j>i}^k \mu a_{k+1,i}a_{k+1,j} \expec{\phi_{2^i\mu}\phi_{2^j\mu}} .\label{eq:pr-prop2-4.2}
\end{eqnarray}
We then prove  \eqref{eq:new-fo-PDE} by induction, recalling that 
\begin{equation*}
A_{\mu,1}\,=\,\expec{(\xi+\nabla \phi_{\mu})\cdot A (\xi+\nabla \phi_{\mu})}.
\end{equation*}
Let us assume that \eqref{eq:new-fo-PDE}  holds at step $k\geq 1$.
Combined with \eqref{eq:pr-prop2-4.1} \& \eqref{eq:pr-prop2-4.2}, \eqref{eq:pr-prop2-1.1} turns into
\begin{eqnarray*}
\xi\cdot A_{\mu,k+1}\xi &=&\xi\cdot A_{\mu,k}\xi-2^{k(k-1)}\mu^{2k}\mathcal{E}(\mf{d}_{\mu,k+1},\mf{d}_{\mu,k+1})-2^{k^2+1}\mu^{2k+1}\expec{\mf{d}_{\mu,k+1}^2} \\
&=&\expec{(\xi+\nabla \phi_{\mu})\cdot A (\xi+\nabla \phi_{\mu})} 
\nonumber \\
&&\qquad \qquad + \mu\sum_{i=0}^{k-1}\eta_{k,i}\expec{\phi_{2^i\mu}^2}+ \mu\sum_{i=0}^{k-1}\sum_{j>i}^{k-1}\nu_{k,i,j}\expec{\phi_{2^i\mu}\phi_{2^j\mu}} \\
&& + 2^{k(k-1)}  \mu \sum_{i=0}^k\sum_{j>i}^k(2^i+2^j)a_{k+1,i}a_{k+1,j}\expec{\phi_{2^i\mu}\phi_{2^j\mu}}+ 2^{k(k-1)} \mu \sum_{i=0}^k 2^i a_{k+1,i}^2\expec{\phi_{2^i\mu}^2}
 \\
&&
-2^{k^2 + 1} \mu \sum_{i=0}^k a_{k+1,i}^2\expec{\phi_{2^i\mu}^2}-2^{k^2 + 2}\mu \sum_{i=0}^k\sum_{j>i}^k a_{k+1,i}a_{k+1,j} \expec{\phi_{2^i\mu}\phi_{2^j\mu}} ,
\end{eqnarray*}
from which we deduce that \eqref{eq:new-fo-PDE}  holds at step $k+1$.
\end{proof}

\medskip
\noindent Proposition~\ref{prop:new-fo-PDE} yields the following formulas for the first four approximations of $A_\ho$:
\begin{eqnarray*}
\xi\cdot A_{\mu,1}\xi&=& \expec{(\xi+\nabla \phi_{\mu})\cdot A (\xi+\nabla \phi_{\mu})},\\
\xi\cdot A_{\mu,2}\xi&=& \expec{(\xi+\nabla \phi_{\mu})\cdot A (\xi+\nabla \phi_{\mu})}-3\mu \expec{\phi_{\mu}^2}-2\mu \expec{\phi_{2\mu}^2}+5\mu \expec{\phi_{\mu}\phi_{2\mu}},\\
\xi\cdot A_{\mu,3}\xi&=& \expec{(\xi+\nabla \phi_{\mu})\cdot A (\xi+\nabla \phi_{\mu})}-\frac{55}{9}\mu \expec{\phi_{\mu}^2}-8\mu  \expec{\phi_{2\mu}^2}-\frac{4}{9}\mu \expec{\phi_{4\mu}^2}\\
&& \qquad \qquad \qquad \qquad +\frac{41}{3}\mu \expec{\phi_{\mu}\phi_{2\mu}}-\frac{22}{9}\mu \expec{\phi_{\mu}\phi_{4\mu}}+\frac{10}{3}\mu \expec{\phi_{2\mu}\phi_{4\mu}},\\
\xi\cdot A_{\mu,4}\xi&=& \expec{(\xi+\nabla \phi_{\mu})\cdot A (\xi+\nabla \phi_{\mu})}-\frac{3655}{441}\mu \expec{\phi_{\mu}^2}-\frac{128}{9}\mu  \expec{\phi_{2\mu}^2}-\frac{16}{9}\mu \expec{\phi_{4\mu}^2} \\
&&\qquad-\frac{8}{441}\mu \expec{\phi_{8\mu}^2}+\frac{1325}{63}\mu \expec{\phi_{\mu}\phi_{2\mu}}-\frac{370}{63}\mu \expec{\phi_{\mu}\phi_{4\mu}}+\frac{184}{441}\mu \expec{\phi_{\mu}\phi_{8\mu}} \\
&&\qquad \qquad \qquad \qquad +\frac{82}{9}\mu \expec{\phi_{2\mu}\phi_{4\mu}}-\frac{44}{63}\mu \expec{\phi_{2\mu}\phi_{8\mu}}+\frac{20}{63}\mu \expec{\phi_{4\mu}\phi_{8\mu}}.
\end{eqnarray*}

\subsection{Complete error estimate}

In this subsection, we combine the approximation formulas $A_{\mu,k}$ with the filtering method
used in \cite{Gloria-09}. The filters are defined as follows.
\begin{defi}\label{def:mask}
A function $\chi:[-1,1]\to \R^+$ is said to be a filter of order $p \geq 0$ if
\begin{itemize}
\item[(i)] $\chi \in C^{p}([-1,1])\cap W^{p+1,\infty}((-1,1))$,
\item[(ii)] $\int_{-1}^1 \chi(x)\d x=1$,
\item[(iii)] $\chi^{(k)}(-1)=\chi^{(k)}(1)=0$ for all $k\in \{0,\dots,p-1\}$.
\end{itemize}
The associated mask $\chi_L:[-L,L]^d\to \R^+$ in dimension $d\geq 1$ is then defined for all $L>0$ by
\begin{equation*}
\chi_L(x)\,:=\,L^{-d}\prod_{i=1}^d\chi(L^{-1}x_i),
\end{equation*}
where $x=(x_1,\dots,x_d)\in \R^d$.
\end{defi}
\noindent Let now $A$ and $A_\ho$ be as in Definition~\ref{defi:A-per}.
For all $k\geq 1$, $\mu>0$, $p\geq 0$, and $R\geq L>0$, we define the approximation $A_{\mu,k,R,L}$ of $A_\ho$
as
\begin{eqnarray}\label{eq:AkmuRL}
\xi\cdot A_{\mu,k,R,L}\xi&:=&\langle\!\langle (\xi+\nabla \phi_{\mu,R})\cdot A (\xi+\nabla \phi_{\mu,R}) \rangle\!\rangle_L \nonumber \\
&& +\mu\sum_{i=0}^{k-1}\eta_{k,i}\langle\!\langle\phi_{2^i\mu,R}^2 \rangle\!\rangle_L+\mu\sum_{i=0}^{k-1}\sum_{j>i}^{k-1}\nu_{k,i,j}\langle\!\langle\phi_{2^i\mu,R}\phi_{2^j\mu,R}\rangle\!\rangle_L,
\end{eqnarray}
where the coefficients $\eta_{k,i}$ and $\nu_{k,i,j}$ are as in Proposition~\ref{prop:new-fo-PDE}, the modified correctors $\phi_{2^i\mu,R}$ are
the unique weak solutions in $H^1_0(Q_R)$ to 
\begin{equation*}
2^i\mu \phi_{2^i\mu,R}-\nabla \cdot A(\xi+\nabla \phi_{2^i\mu,R})\,=\,0,
\end{equation*}
and $\langle\!\langle \cdot \rangle\!\rangle_L $ denotes the average with mask $\chi_L$:
\begin{equation*}
\langle\!\langle h \rangle\!\rangle_L\,:=\,\int_{\R^d}h(x)\chi_L(x)\d x.
\end{equation*}
The combination of \cite[Theorem~1]{Gloria-09} with Theorem~\ref{th:main-perio} and Proposition~\ref{prop:new-fo-PDE} then yields
\begin{theo}\label{th:error-estim-per}
Let $d\geq 2$, $A$ and $A_\ho$ be as in Definition~\ref{defi:A-per}, $k\geq 1$, $\chi$ be a filter of order $p\geq 0$, and $A_{\mu,k,R,L}$ be the
approximation \eqref{eq:AkmuRL} of the homogenization matrix, where
$R^2 \gtrsim \mu^{-1} \gtrsim R$, $R\geq L \sim R \sim R-L$.
Then, there exists $c>0$ depending only on $\alpha,\beta$ and $d$ such that we have
\begin{equation}\label{eq:th-error-estim-per}
|A_{\mu,k,R,L}-A_\ho|\,\lesssim\, L^{-(p+1)}+\mu^{2k}+\mu^{-1/4} \exp \left(-c \sqrt{\mu}(R-L)\right).
\end{equation}
\end{theo}

\bigskip
\noindent In order to illustrate Theorem~\ref{th:error-estim-per}, we provide the results of numerical 
tests in a periodic discrete case in Appendix~\ref{append-perio}. They confirm the sharpness of the
analysis.

\section{The discrete stochastic case}\label{sec:sto}

We start this section by defining the discrete stochastic model we wish to consider.

\subsection{Notation and preliminaries}

We say that $x, y$ in $\Z^d$ are neighbors, and write $x \sim y$, whenever $|y-x| = 1$. This relation turns $\Z^d$ into a graph, whose set of (non-oriented) edges we will denote by $\mb{B}$.
We now turn to the definition of the associated diffusion coefficients, and their statistics.
\begin{defi}[environment]\label{defi:envi}
Let $\Omega = [\alpha,\beta]^{\mb{B}}$. An element $\o = (\o_e)_{e \in \mb{B}}$ of $\Omega$ is called an \emph{environment}. 
With any edge $e = (x,y) \in \mb{B}$, we associate the \emph{conductance} $\o_{x,y}:=\o_e$ (by construction $\o_{x,y} = \o_{y,x}$). 
Let $\nu$ be a probability measure on $[\alpha,\beta]$. 
We endow $\Omega$ with the product probability measure $\P = \nu ^{\otimes \mb{B}}$. 
In other words, if $\omega$ is distributed according to the measure $\P$, then $(\omega_e)_{e \in \mb{B}}$ are independent random variables of law $\nu$. 
We denote by $L^2(\Omega)$ the set of real square integrable functions on $\Omega$ for the measure $\mathbb{P}$,
and write $\expec{\cdot}$ for the expectation associated with $\P$. 
\end{defi}
\noindent In the framework of Definition~\ref{defi:envi}, we can introduce a notion of stationarity.
\begin{defi}[stationarity]\label{defi:stat}
For all $z\in \Z^d$, we let $\theta_z:\Omega\to \Omega$ be such that for all $\o \in \Omega$ and $(x,y)\in \mb{B}$,
$(\theta_z \ \omega)_{x,y}=\omega_{x+z,y+z}$. This defines an additive action group $\{\theta_z\}_{z\in \Z^d}$ on $\Omega$ which preserves
the measure $\mathbb{P}$, and is ergodic for  $\mathbb{P}$.

\noindent We say that a function $f:\Omega \times \Z^d\to \R$ is stationary if and only if for all $x,z\in\Z^d$ and
$\mathbb{P}$-almost every $\omega\in \Omega$, 
$$
f(x+z,\omega)\,=\,f(x,\theta_z \ \omega).
$$
In particular, with all $f\in L^2(\Omega)$, one may associate the stationary function (still denoted by $f$) $\Z^d\times\Omega\to \R,
(x,\omega) \mapsto f(\theta_x \ \omega)$.
In what follows we will not distinguish between $f\in L^2(\Omega)$ and its stationary extension on $\Z^d\times\Omega$.
\end{defi}
\noindent It remains to define the conductivity matrix on $\Z^d$.
\begin{defi}[conductivity matrix]\label{defi:conduct}
Let $\Omega$, $\mathbb{P}$, and $\{\theta_z\}_{z\in \Z^d}$ be as in Definitions~\ref{defi:envi} and~\ref{defi:stat}.
The stationary diffusion matrix $A:\Z^d\times\Omega\to \mathcal{M}_d(\R)$ is defined by
$$
A(x,\omega) = \dig{\o_{x,x+\ee_i} ,\dots, \o_{x,x+\ee_d}}.
$$
\end{defi}
\noindent For each $\omega\in\Omega$, we may consider the discrete elliptic equation whose operator is 
$$
-\nabla^*\cdot A (\cdot,\omega)\nabla,
$$
where $\nabla$ and $\nabla^*$ are defined for all $u:\Z^d\to \R$ by
\begin{equation}\label{eq:disc-nabla}
\nabla u(x):=\left[  
\begin{array}{l}
u(x+\ee_1)-u(x) \\
\vdots\\
u(x+\ee_d)-u(x)
\end{array}
\right],
\
\nabla^* u(x):=\left[  
\begin{array}{l}
u(x)-u(x-\ee_1) \\
\vdots\\
u(x)-u(x-\ee_d)
\end{array}
\right],
\end{equation}
and the backward divergence is denoted by $\nabla^*\cdot$, as usual.
The standard stochastic homogenization theory for such discrete elliptic operators (see for instance \cite{Kunnemann-83}, \cite{Kozlov-87}) %, and also \cite{Kipnis-Varadhan-86} for the environment viewed by the particle)  
ensures that 
there exist homogeneous and deterministic coefficients $A_\ho$ such that the solution operator of the continuum
differential operator $-\nabla \cdot A_\ho\nabla$ describes $\mathbb{P}$-almost surely 
the large scale behavior of the solution operator
of the discrete differential operator $-\nabla^*\cdot A (\cdot,\omega)\nabla $.
As for the periodic case, the definition of $A_\ho$ involves the so-called correctors $\phi:\Z^d\times \Omega \to \R$, which are solutions (in a sense made precise below) to
the equations
\begin{equation}\label{eq:corr-sto}
-\nabla^*\cdot A(x,\omega)(\xi+\nabla \phi(x,\omega))\,=\,0,\qquad x \in \Z^d,
\end{equation}
for $\xi \in \R^d$.
The following lemma gives the existence and uniqueness of the corrector $\phi$.
\begin{lemma}[corrector]\label{lem:corr}
Let $\Omega$, $\mathbb{P}$, $\{\theta_z\}_{z\in \Z^d}$, and $A$ be as in Definitions~\ref{defi:envi}, \ref{defi:stat}, and~\ref{defi:conduct}.
Then, for all $\xi\in \R^d$, there exists a unique measurable function $\phi:\Z^d\times \Omega\to \R$ such that 
$\phi(0,\cdot)\equiv 0$, $\nabla \phi$ is stationary, $\expec{\nabla \phi}=0$, and $\phi$ solves \eqref{eq:corr-sto} $\mathbb{P}$-almost surely.
Moreover, the symmetric homogenized matrix $A_\ho$ is characterized by
\begin{equation}\label{eq:hom-coeff}
\xi\cdot A_\ho\xi\,=\,\expec{(\xi+\nabla\phi)\cdot A(\xi+\nabla \phi)}.
\end{equation}
\end{lemma}
\noindent As mentioned in the introduction, the standard proof of Lemma~\ref{lem:corr} 
makes use of the regularization of \eqref{eq:corr-sto} by a zero-order term $\mu>0$:
\begin{equation}\label{eq:mod-corr-sto}
\mu\phi_\mu(x,\omega)-\nabla^*\cdot A(x,\omega)(\xi+\nabla \phi_\mu(x,\omega))\,=\,0, \qquad x \in \Z^d.
\end{equation}
\begin{lemma}[modified corrector]\label{lem:mod-corr}
Let $\Omega$, $\mathbb{P}$, $\{\theta_z\}_{z\in \Z^d}$, and $A$ be as in Definitions~\ref{defi:envi}, \ref{defi:stat}, and~\ref{defi:conduct}.
Then, for all $\mu>0$ and $\xi\in \R^d$, there exists a unique stationary function $\phi_\mu \in L^2(\Omega)$ 
which solves \eqref{eq:mod-corr-sto} $\mathbb{P}$-almost surely.
\end{lemma}
\noindent In order to proceed as in the periodic case and use a spectral approach, one needs to suitably define an elliptic operator
on $L^2(\Omega)$ (which is the stochastic counterpart to the space $H^1_\per(Q)$ of Section~\ref{sec:perio}).
Stationarity is crucial here. Following \cite{Papanicolaou-Varadhan-79}, we introduce differential operators on $L^2(\Omega)$: 
for all $u\in L^2(\Omega)$, we set
\begin{equation}\label{eq:disc-nabla-sto}
%\overline{\nabla} 
\D u(\omega):=\left[  
\begin{array}{l}
u(\theta_{\ee_1}\omega)-u(\omega) \\
\vdots\\
u(\theta_{\ee_d})-u(\omega)
\end{array}
\right],
\
%\overline{\nabla}^* 
\D^*u(\omega):=\left[  
\begin{array}{l}
u(\o)-u(\theta_{-\ee_1}\o) \\
\vdots\\
u(\o)-u(\theta_{-\ee_d}\o)
\end{array}
\right].
\end{equation} 
We are in position to define the stochastic counterpart to the operator of Definition~\ref{defi:decomp-spec-perio}.
\begin{defi}\label{defi:operator-sto}
Let $\Omega$, $\mathbb{P}$, $\{\theta_z\}_{z\in \Z^d}$, and $A$ be as in Definitions~\ref{defi:envi}, \ref{defi:stat}, and~\ref{defi:conduct}.
We define $\L:L^2(\Omega)\to L^2(\Omega)$ by
\begin{eqnarray*}
\L u(\omega)&=&-\D^*\cdot A(\omega) \D u(\omega) \\
%-\overline{\nabla}^*\cdot A(\omega) \overline{\nabla}u(\omega) \\
&=&\sum_{z \sim 0} \o_{0,z} (u(\o) -u(\theta_z \ \o))
\end{eqnarray*}
%
%where $\overline{\nabla}$ and $\overline{\nabla}^*$ are as in \eqref{eq:disc-nabla-sto}.
where $\D$ and $\D^*$ are as in \eqref{eq:disc-nabla-sto}.
\end{defi}
\noindent In probabilistic terms, the operator $-\L$ is the generator of the Markov process called the ``environment viewed by the particle''. This process is defined to be $(\theta_{X_t} \ \o)$, where $(X_t)$ is a random walk whose jump rate from $x$ to a neighbor $y$ is given by $\o_{x,y}$. 

\medskip
\noindent Using Definition~\ref{defi:operator-sto} and the stationarity of $\phi_\mu$, Lemma~\ref{lem:mod-corr}
implies that $\phi_\mu$ is the unique solution in $L^2(\Omega)$ to the equation
\begin{equation}\label{eq:mod-corr-Om}
(\mu+\L)\phi_\mu \,=\,\mf{d},
\end{equation}
where 
\begin{equation}\label{eq:loc-drift-sto}
%\mf{d}(\omega)\,:=\,\overline{\nabla}^*\cdot A(\omega)\xi.
\mf{d}(\omega)\,:=\,\D^*\cdot A(\omega)\xi.
\end{equation}
At the level of the corrector $\phi$ itself (which is not stationary), the weak form of \eqref{eq:mod-corr-Om} survives for $\mu=0$: for
every $\psi\in L^2(\Omega)$, we have
\begin{equation}\label{eq:mod-corr-Om-weak}
%\expec{\overline{\nabla}\psi \cdot A\overline{\nabla}\phi} \,=\,\expec{\overline{\nabla}\psi\cdot A\xi}.
\expec{\D\psi \cdot A\D\phi} \,=\,\expec{\D\psi\cdot A\xi}.
\end{equation}

\medskip
\noindent
For all $f\in L^2(\Omega)$, we let $\cE(f,f)$ be the Dirichlet form associated with $\L$, defined by
\begin{equation}
\label{defEff}
%\cE(f,f) = \la \L f \cdot f \ra = \la  \overline{\nabla} f \cdot A\overline{\nabla} f \ra = \frac{1}{2}  \sum_{z \sim 0} \left\la \o_{0,z} (f(\theta_z \ \o) - f(\o))^2 \right\ra.
\cE(f,f) = \la \L f \cdot f \ra = \la  \D f \cdot A\D f \ra = \frac{1}{2}  \sum_{z \sim 0} \left\la \o_{0,z} (f(\theta_z \ \o) - f(\o))^2 \right\ra.
\end{equation}
As in the periodic case, the homogenized diffusion matrix satisfies 
the identity
\begin{equation}
\label{Ahdirichlet}
\xi \cdot \Ah \xi = \la \xi \cdot A \xi \ra - \cE(\phi,\phi).
\end{equation}
The proof is formally the same as for \eqref{eq:hom-Dir}, provided 
we use the weak form \eqref{eq:mod-corr-Om-weak} of the corrector equation,
which holds for $\phi$ in place of $\psi$ (although $\phi$ is not stationary).

\medskip
\noindent We refer the reader to \cite{Kunnemann-83} for the proofs of the statements above.

\subsection{Spectral representation and approximations of the homogenized coefficients}

The operator $\L$ is bounded, positive, and self-adjoint on $L^2(\Omega)$. 
By the spectral theorem,
for any function $f \in L^2(\Omega)$, we can define the spectral measure $e_f$ of $\L$ projected on $f$, that is such that for any bounded continuous function $\Psi : \R_+ \to \R$, one has
$$
\la f \cdot \Psi(\L) f \ra = \int_{\R^+}\Psi(\lambda) \ \d e_f(\lambda).
$$
As in the periodic case, we can express the homogenized diffusion matrix in terms of the spectral measure projected on $\mf{d}$. 
\begin{lemma}\label{lem:spec-decomp}
Let $\Omega$, $\mathbb{P}$, $\{\theta_z\}_{z\in \Z^d}$, $A$, and $\L$ be as in Definitions~\ref{defi:envi}, \ref{defi:stat}, \ref{defi:conduct},
and~\ref{defi:operator-sto}. We let $A_\ho$ denote the associated homogenized diffusion matrix \eqref{eq:hom-coeff}, and $\mf{d}$
be the local drift \eqref{eq:loc-drift-sto}.
Then, the following identity holds
$$
\xi \cdot \Ah \xi = \la \xi \cdot A \xi \ra - \int_{\R^+}\frac{1}{\lambda} \ \d e_\mf{d}(\lambda),
$$
where $e_\mf{d}$ is the spectral measure of $\L$ projected on $\mf{d}$. 
\end{lemma}
\begin{proof}
In view of formula~\eqref{Ahdirichlet}, we need to show that
$$
\cE(\phi,\phi) = \int_{\R^+}\frac{1}{\lambda} \ \d e_\mf{d}(\lambda).
$$
This is either a consequence of Kipnis and Varadhan's arguments (see in particular \cite[Theorem~8.1]{Mourrat-10}),
or a consequence of \cite[Corollary~1 \& Remark~2]{Gloria-Otto-09b}.
We detail the second argument. \cite[Corollary~1 \& Remark~2]{Gloria-Otto-09b} imply that $\lim_{\mu\to 0}\nabla \phi_\mu = \nabla \phi$ strongly in $L^2(\Omega)$, hence
$$
\lim_{\mu\to 0} \cE(\phi_\mu,\phi_\mu) \,=\, \cE(\phi,\phi).
$$
Besides, for all $\mu>0$, we have by definition of the spectral decomposition 
$$
\cE(\phi_\mu,\phi_\mu) = \int_{\R^+}\frac{\lambda}{(\lambda+\mu)^2} \ \d e_f(\lambda),
$$
and the result follows by the monotone convergence theorem.
\end{proof}

\noindent From Lemma~\ref{lem:spec-decomp}, we deduce that the approximations $A_{\mu,k}$
introduced in Theorem~\ref{th:main-perio} and further characterized in Proposition~\ref{prop:new-fo-PDE}
may also be used in this discrete stochastic case, provided the notation $\expec{\cdot}$ is understood
as the expectation (instead of periodic average). 

\subsection{Suboptimal estimate of the systematic error}
We let $\mf{d}_{\mu,k}$, $P_k$, and $A_{\mu,k}$ be as in Section~\ref{sec:perio}.
In order to quantify the systematic error, we introduce, for any $D,q,k \ge 0$, the function $\mathrm{Err}_{D,q,k} : \R_+ \to \R$ defined by
$$
\mathrm{Err}_{D,q,k}(\mu) = 
\left|
\begin{array}{ll}
\mu^{2k} & \text{if } k < D/4, \\
\mu^{2k} \ln_+^{1+q}(\mu^{-1}) & \text{if } k = D/4, \\
\mu^{D/2} \ln_+^q(\mu^{-1}) & \text{if } k > D/4,
\end{array}
\right.
$$
where we write $\ln_+(x) = \max \{\ln x, 1\}$. The purpose of this section is to show the following theorem.
\begin{theo}\label{theo:main-sto}
Let $\Omega$, $\mathbb{P}$, $\{\theta_z\}_{z\in \Z^d}$, $A$, and $\L$ be as in Definitions~\ref{defi:envi}, \ref{defi:stat}, \ref{defi:conduct},
and~\ref{defi:operator-sto}, and $e_{\mf{d}}$ be as in Lemma~\ref{lem:spec-decomp}. 
We let $A_\ho$ denote the associated homogenized diffusion matrix \eqref{eq:hom-coeff}, and $A_{\mu,k}$
be the approximation \eqref{eq:th1-def-app-A} of $A_\ho$ 
for $\mu>0$, and $k\geq 1$.
Then, there exists $q \ge 0$ (depending on $\alpha$ and $\beta$) such that for all $\xi\in \R^d$ with $|\xi|=1$,
$$
0 \le \xi \cdot (A_{\mu,k} - \Ah) \xi \lesssim \left|
\begin{array}{ll}
\mathrm{Err}_{2,q,k}(\mu) & \text{if } d = 2, \\
\mathrm{Err}_{d,0,k}(\mu) & \text{if } 5\geq d>2, \\
\mathrm{Err}_{6,1,k}(\mu) & \text{if } d = 6, \\
\mathrm{Err}_{6,0,k}(\mu) & \text{if } 12\geq d>6, \\
\mathrm{Err}_{d-6,0,k}(\mu) & \text{if } d > 12.
\end{array}
\right.
$$
\end{theo}
\noindent In order to prove Theorem~\ref{theo:main-sto}, we need to introduce some vocabulary. 
For all $\gamma > 1$ and $q \ge 0$, we say that the spectral exponents of a function $f \in L^2(\Omega)$ are at least $(\gamma,-q)$ if we have
$$
\int_0^\mu \d e_f(\lambda) \lesssim  \mu^\gamma \ln_+^{q}(\mu^{-1}).
$$
Note that, if $(\gamma',-q') \le (\gamma,-q)$ for the lexicographical order, and if the spectral exponents of $f$ are at least $(\gamma,-q)$, then they are at least $(\gamma',-q')$.
Hence, the phrasing is consistent.

\medskip
\noindent In order to prove Theorem~\ref{theo:main-sto}, we first express the systematic error in terms of the spectral exponents of $\mf{d}$. 
This is the object of Theorem~\ref{theo:sto-2}.
We then prove estimates on these exponents in Theorem~\ref{theo:exponents}, which concludes the proof of  Theorem~\ref{theo:main-sto}.
\begin{theo}
\label{theo:sto-2}
Within the notation and assumptions of Theorem~\ref{theo:main-sto}, the following two statements hold: 
for all $\xi \in \R^d$ with $|\xi|=1$,
\begin{enumerate}
\item
\label{eq:item-thmsto1}
If the spectral exponents of $\mf{d}$ are at least $(\gamma,-q)$, then 
$$
0 \le \xi \cdot (A_{\mu,k} - \Ah) \xi \,\lesssim \, \left|
\begin{array}{ll}
\mu^{2k} & \text{if } \gamma > 2k+1, \\
\mu^{2k} \ \ln_+^{1+q}(\mu^{-1}) & \text{if } \gamma = 2k+1, \\
\mu^{\gamma - 1} \ \ln_+^{q}(\mu^{-1}) & \text{if } \gamma < 2k+1.
\end{array}
\right.
$$
\item
\label{eq:item-thmsto2}
Conversely, 
$$
\xi \cdot (A_{\mu,k} - \Ah) \xi \,\gtrsim \,\mu^{2k} + \mu^{-1} \int_0^{\mu} \d e_\mf{d}(\lambda).
$$
\end{enumerate}
\end{theo}
\noindent This theorem extends \cite[Proposition~9.1]{Mourrat-10}. We begin by proving the following result.
\begin{lemma} \label{lem:ipp}
If the spectral exponents of $\mf{d}$ are at least $(\gamma,-q)$, then 
\begin{equation}
\label{eq:ipp}
0 \le \xi \cdot (A_{\mu,k} - \Ah) \xi \lesssim \mu^{2k} + \mu^{\gamma-1} \int_0^{\mu^{-1}} \frac{u^{\gamma-2}}{(1+u)^{2k}} \ln^{q}((\mu u)^{-1}) \ \d u.
\end{equation}
\end{lemma}
\begin{proof}[Proof of Lemma~\ref{lem:ipp}]
First, recall that
$$
\xi \cdot (A_{\mu,k} - \Ah)\xi = 2^{k(k-1)} \mu^{2k} \int_{\R^+}\frac{1}{\lambda (\mu+\lambda)^2 \cdots (2^{k-1} \mu + \lambda)^2} \ \d e_{\mf{d}}(\lambda) .
$$
The integral of the r.~h.~s. is non-negative and bounded by
\begin{equation}
\label{intmaj}
\int_{\R^+}\frac{1}{\lambda (\mu + \lambda)^{2k}} \ \d e_{\mf{d}}(\lambda).
\end{equation}
We perform a sort of integration by parts on this integral. 
To this aim, we let $f'(\lambda)$ be given by
$$
f'(\lambda) = - \frac{\partial}{\partial \lambda} \frac{1}{\lambda (\mu+\lambda)^{2k}} = \frac{(\mu+\lambda)^{2k-1} (\mu+(2k+1)\lambda)}{\lambda^2(\mu+\lambda)^{4k}}.
$$
We then rewrite the integral (\ref{intmaj}) in terms of $f'$, and use Fubini's theorem:
\begin{eqnarray*}
\int_{\R^+}\frac{1}{\lambda (\mu + \lambda)^{2k}} \ \d e_{\mf{d}}(\lambda) & = & \int_{\lambda = 0}^{+\infty} \int_{\delta = \lambda}^{+\infty} f'(\delta) \ \d \delta \ \d e_{\mf{d}}(\lambda) \\
& = & \int_{\delta = 0}^{+\infty} f'(\delta) \int_{\lambda = 0}^{\delta} \d e_{\mf{d}}(\lambda) \ \d \delta.
\end{eqnarray*}
We split this double integral in two parts, and treat the cases $\delta \in (1,+\infty)$
and $\delta\in (0,1]$ separately.
We begin with the case when $\delta$ ranges in $(1,+\infty)$.
We bound the inner integral 
$$
\int_{\lambda = 0}^{\delta} \d e_{\mf{d}}(\lambda) \,\leq \,\int_{\lambda = 0}^{\infty} \d e_{\mf{d}}(\lambda)\,=\,\expec{\mf{d}^2}\,\leq \, 4\beta^2\,\lesssim\, 1,
$$
by definition of the projection of the spectral measure on $\mf{d}$.
This yields for the first part of the double integral
$$
\int_{\delta = 1}^{+\infty} f'(\delta) \int_{\lambda = 0}^{\delta} \d e_{\mf{d}}(\lambda) \ \d \delta\,\lesssim\, \frac{1}{(\mu + 1)^{2k}} \,\lesssim\, 1.
$$
We now turn to the case when $\delta$ ranges in $(0,1]$.
The assumption on the spectral exponents of $\mf{d}$ implies
\begin{equation}
\label{intamaj}
\int_{\delta =0}^{1} f'(\delta) \int_{\lambda = 0}^{\delta} \d e_{\mf{d}}(\lambda) \ \d \delta \,\leq \,\int_{0}^{1} f'(\delta) \delta^\gamma \ln^{q}(\delta^{-1}) \ \d \delta.
\end{equation}
Noting that
$$
f'(\delta) \le (2k+1) \frac{1}{\delta^2 (\mu + \delta)^{2k}},
$$
we bound the r.~h.~s. of (\ref{intamaj}) by $(2k+1)$ times
$$
\int_{0}^{1} \frac{\delta^{\gamma-2}}{(\mu + \delta)^{2k}}  \ln^{q}(\delta^{-1}) \ \d \delta.
$$
A change of variables yields the announced result.
\end{proof}
\begin{proof}[Proof of part~(\ref{eq:item-thmsto1}) of Theorem~\ref{theo:sto-2}]
We first assume that $\gamma > 2k+1$. 
In that case, we let $\gamma'$ be such that $2k+1 < \gamma' < \gamma$. 
Since the spectral exponents of $\mf{d}$ are at least $(\gamma',0)$,  Lemma~\ref{lem:ipp} ensures that
$$
0 \le \xi \cdot (A_{\mu,k} - \Ah) \xi \,\lesssim\, \mu^{2k} + \mu^{\gamma'-1} \int_0^{\mu^{-1}} \frac{u^{\gamma'-2}}{(1+u)^{2k}} \ \d u \,\lesssim \, \mu^{2k}.
$$
We now turn to the case when $\gamma \le 2k+1$. 
We need to estimate the integral of the r.~h.~s. of (\ref{eq:ipp}). To this aim, we note that
$$
\ln^{q}((\mu u)^{-1}) = \big( \ln(\mu^{-1}) - \ln(u) \big)^{q} \le 2^{q}\big( \ln^{q}(\mu^{-1}) + \left|\ln(u)\right|^{q} \big),
$$
so that the integral in (\ref{eq:ipp}) may be estimated by
$$
\int_0^{\mu^{-1}} \frac{u^{\gamma-2}}{(1+u)^{2k}}\big( \ln^{q}(\mu^{-1}) + \left|\ln(u)\right|^{q} \big) \ \d u \,\lesssim\,
\left|
\begin{array}{ll}
\ln^{q+1}(\mu^{-1}) & \text{if } \gamma = 2k+1, \\
\ln^{q}(\mu^{-1}) & \text{if } \gamma < 2k+1,
\end{array}
\right.
$$
as desired.
\end{proof}
\begin{proof}[Proof of part~(\ref{eq:item-thmsto2}) of Theorem~\ref{theo:sto-2}]
Let $\delta > 0$ be such that
$$
\int_0^\delta \d e_\mf{d}(\lambda) > 0.
$$
By the non-negativity of the spectrum and of the integrand,
\begin{eqnarray*}
\xi \cdot (A_{\mu,k} - \Ah)\xi & = & \int_{\R^+}\frac{2^{k(k-1)} \mu^{2k}}{\lambda (\mu+\lambda)^2 \cdots (2^{k-1} \mu + \lambda)^2} \ \d e_{\mf{d}}(\lambda) \\
& \ge & \frac{2^{k(k-1)} \mu^{2k}}{\delta (\mu+\delta)^2 \cdots (2^{k-1} \mu + \delta)^2} \int_0^\delta\d e_\mf{d}(\lambda).
\end{eqnarray*}
Hence, 
$$
\xi \cdot (A_{\mu,k} - \Ah)\xi \,\gtrsim\, \mu^{2k}.
$$
In addition, there exists $C > 0$ such that for all $\lambda \in (0,\mu]$, one has
$$
\frac{\mu^{2k}}{\lambda (\mu+\lambda)^2 \cdots (2^{k-1} \mu + \lambda)^2} \ge \frac{C}{\mu}.
$$
Therefore, 
$$
\xi \cdot (A_{\mu,k} - \Ah)\xi \gtrsim \mu^{-1} \int_0^\mu \d e_\mf{d}(\lambda),
$$
which concludes the proof of the theorem.
\end{proof}
\noindent It remains to estimate the spectral exponents of $\mf{d}$.
\begin{theo}
\label{theo:exponents}
Within the notation and assumptions of Theorem~\ref{theo:main-sto},  there exists $q \ge 0$
depending only on the ellipticity constants $\alpha$ and $\beta$
 such that the spectral exponents of $\mf{d}$ are at least
$$
\left|
\begin{array}{ll}
(2,-q) & \text{if } d = 2,\\
(d/2 + 1,0) & \text{if } 5 \ge d > 2 ,\\
(4,-1) & \text{if } d = 6,\\
(4,0) & \text{if } 12 \ge d >6,\\
(d/2-2,0) & \text{if } d > 12.
\end{array}
\right.
$$
\end{theo}
\begin{rem} \label{rem:edge}
We conjecture that the spectral exponents of $\mf{d}$ are in fact $(d/2+1,0)$ for $d>2$. 
If true, this would imply that the systematic error is in fact bounded by $\mathrm{Err}_{d,0,k}(\mu)$ for any $d>2$ and $k$. 
\end{rem}
\noindent In order to prove Theorem~\ref{theo:exponents}, we will make use of the following result.
\begin{lemma}
\label{lem:app}
Within the notation and assumptions of Theorem~\ref{theo:main-sto}, there exists $q \ge 0$ depending only on 
the ellipticity constants $\alpha$ and $\beta$ such that
$$
\la (\mf{d}_{\mu,2})^2 \ra = \int_{\R^+}\frac{1}{(\mu + \lambda)^2(2\mu + \lambda)^2} \ \d e_\mf{d}(\lambda) \lesssim 
\left|
\begin{array}{ll}
\mu^{-2} \ln_+^q(\mu^{-1}) & \text{if } d = 2, \\
\mu^{d/2 - 3}  & \text{if } 5 \ge  d >2, \\
\ln_+(\mu) & \text{if } d = 6, \\
1 & \text{if } d >6,
\end{array}
\right.
$$
where $\mf{d}_{\mu,2}$ is as in Definition~\ref{def:mfd-mu-k}.
\end{lemma}
\noindent Lemma~\ref{lem:app} is a consequence of the results of  \cite{Gloria-Otto-09b}.
Its proof, which is slightly technical, is deferred to Appendix~\ref{append}. 
\begin{proof}[Proof of Theorem~\ref{theo:exponents}]
For all $\lambda \le \mu$, one has
$$
\frac{\mu^4}{(\mu + \lambda)^2(2\mu + \lambda)^2} \ge \frac{1}{36}.
$$
Hence,
$$
\int_0^\mu \d e_\mf{d}(\lambda) \le 36 \mu^4 \int_{\R^+}\frac{1}{(\mu + \lambda)^2(2\mu + \lambda)^2} \ \d e_\mf{d}(\lambda).
$$
The announced bounds then follow from Lemma~\ref{lem:app} for $d \le 12$.

\medskip
\noindent For $d \ge 13$, we use instead \cite[Theorems 2.3 and 2.4]{Mourrat-10}, which ensure 
that there exist $C > 0$ such that for all $\mu>0$,
$$
\int_0^\mu \lambda^{-1} \d e_\mf{d}(\lambda) \le C \mu^{d/2 - 3}.
$$
This shows that the spectral exponents of $\mf{d}$ are at least $(d/2-2,0)$, since
$$
\int_0^\mu \d e_\mf{d}(\lambda) \le \mu \int_0^\mu \lambda^{-1} \d e_\mf{d}(\lambda).
$$
\end{proof}

\subsection{Complete error analysis}

As for the periodic case, $\phi_\mu$ can be accurately replaced by $\phi_{\mu,R}$, the solution 
of the modified corrector equation on a finite box $Q_R$ with homogeneous Dirichlet
boundary conditions. We refer the reader to \cite{Gloria-10} for details.

\medskip
\noindent In order to perform a complete error estimate, one still needs to estimate the variance term in the r.~h.~s. 
of the identity corresponding to \eqref{intro:error-estim}.
This is the object of the following theorem.
\begin{theo}\label{th:variance+}
Let $\Omega$, $\mathbb{P}$, $\{\theta_z\}_{z\in \Z^d}$, and $A$ be as in Definitions~\ref{defi:envi}, \ref{defi:stat}, and \ref{defi:conduct}.
We let $A_\ho$ denote the associated homogenized diffusion matrix \eqref{eq:hom-coeff}, and 
for all $k\geq 1$, $\mu>0$, and $L>0$, we define the approximation $A_{\mu,k,L}$ of $A_\ho$
as
\begin{eqnarray*}
\xi\cdot A_{\mu,k,L}\xi&:=&\langle\!\langle (\xi+\nabla \phi_{\mu})\cdot A (\xi+\nabla \phi_{\mu}) \rangle\!\rangle_L \nonumber \\
&& +\mu\sum_{i=0}^{k-1}\eta_{k,i}\langle\!\langle\phi_{2^i\mu}^2 \rangle\!\rangle_L+\mu\sum_{i=0}^{k-1}\sum_{j>i}^{k-1}\nu_{k,i,j}\langle\!\langle\phi_{2^i\mu}\phi_{2^j\mu}\rangle\!\rangle_L,
\end{eqnarray*}
where the coefficients $\eta_{k,i}$ and $\nu_{k,i,j}$ are as in Proposition~\ref{prop:new-fo-PDE}, the modified correctors $\phi_{2^i\mu}$ are
as in Lemma~\ref{lem:mod-corr},  and $\langle\!\langle \cdot \rangle\!\rangle_L $ denotes the spatial average
\begin{eqnarray*}
h\mapsto \langle\!\langle h \rangle\!\rangle_L\,:=\,\int_{\Z^d}h(x)\chi_L(x)\d x,
\end{eqnarray*}
where $x\mapsto \chi_L(x)$ is an averaging function on $(-L,L)^d$ such that $\int_{\Z^d}\chi_L(x) \d x=1$ and $\|\nabla \chi_L\|_{L^\infty} \lesssim L^{-d-1}$. 
Then, there exists an exponent $q>0$ depending only on $\alpha,\beta$ such that
\begin{equation*}
\var{A_{\mu,k,L}} \,\lesssim \,\left|
\begin{array}{rcl}
(L^{-2}+\mu^2)\ln_+^q \mu^{-1} & \text{if} & d=2,\\
L^{-d}+\mu^2 L^{-d+2} & \text{if} &d>2.
\end{array}
\right.
\end{equation*}
\end{theo}
\noindent Theorem~\ref{th:variance+} is a direct consequence of \cite[Theorem~2.1 \& Remark~2.1]{Gloria-Otto-09} applied
to each term of $A_{\mu,k}$ in the form \eqref{eq:new-fo-PDE} of Proposition~\ref{prop:new-fo-PDE}.

\subsection{Polynomial decay of the variance along the semi-group}
We end this section with a short remark concerning some results of \cite{Mourrat-10}. Let $(S_t)_{t \ge 0}$ be the semi-group associated with the infinitesimal generator $-\L$ introduced in Definition~\ref{defi:operator-sto}. In \cite{Mourrat-10}, the asymptotic decay to $0$ of the variance of $S_t f$ is investigated. A slight modification of \cite[Theorem~2.4]{Mourrat-10} reads as follows.
\begin{theo}
\label{decay-spectrum}
Let $f \in L^2(\Omega)$ be such that $\langle f \rangle = 0$, and let $\gamma > 1$, $q \ge 0$. The following two statements are equivalent~:
\begin{enumerate}
\item
The spectral exponents of $f$ are at least $(\gamma,-q)$ ;
\item
$$
\left\langle (S_t f)^2 \right\rangle \lesssim t^{-\gamma} \ \ln_+^q(t) .
$$
\end{enumerate}
\end{theo}
\noindent From Theorem~\ref{theo:exponents}, we thus obtain the following result, which strengthens \cite[Theorem~2.3 and Corollary~9.3]{Mourrat-10} when $4 \le d < 12$.
\begin{corollary}
Within the notation and assumptions of Theorem~\ref{theo:main-sto}, there exists $q \ge 0$
depending only on the ellipticity constants $\alpha$ and $\beta$ such that
$$
\left\langle (S_t \mf{d})^2\right\rangle \lesssim 
\left|
\begin{array}{ll}
t^{-2} \ \ln_+^q(t)  & \text{if } d = 2,\\
t^{-(d/2 + 1)} & \text{if } 5 \ge d > 2 ,\\
t^{-4} \ \ln_+(t)  & \text{if } d = 6,\\
t^{-4} & \text{if } 12 \ge d >6,\\
t^{-(d/2-2)} & \text{if } d > 12.
\end{array}
\right.
$$
\end{corollary}

\section*{Acknowledgements}

\noindent The authors acknowledge the support of INRIA, through the grant ``Action de
Recherche Collaborative'' DISCO.

\appendix

\section{Proof of Lemma~\ref{lem:app}}\label{append}

\noindent We adopt the notation of \cite{Gloria-Otto-09b}.
In particular, we set $T=\mu^{-1}$, denote by $G_T$ the Green's function associated
with the elliptic operator $T^{-1}-\nabla^* \cdot A\nabla$, $\phi_T$ the associated modified
corrector, and we set $\psi_T:=\mf{d}_{\mu,2}$. Note that $G_T$ and $\phi_T$ depend on the 
diffusion coefficients $A$.
The claim of the lemma is equivalent to 
\begin{equation}
\expec{\psi_T^2} \,\lesssim \,
\left|
\begin{array}{rcl}
T^2\ln^q T & \text{if} & d=2,\\
T^{3-d/2} & \text{if} & 5\geq d >2,\\
\ln T &\text{if} & d=6, \\
1 &\text{if} &d>6.
\end{array}
\right.
\end{equation}
Since $\expec{\psi_T}=0$, it holds that $\expec{\psi_T^2}=\var{\psi_T}$.
From the identity $\psi_T=T(\phi_T-\phi_{2T})$ we learn that $\psi_T$ depends continuously on the diffusion coefficients by 
\cite[Lemma~2.6]{Gloria-Otto-09} so that one may apply the variance estimate of \cite[Lemma~2.3]{Gloria-Otto-09}.
In particular, 
\begin{equation}
\var{\psi_T} \,\lesssim \, \sum_{e}\expec{\sup_{\omega_e}\left(\frac{\partial \psi_T(0)}{\partial \omega_e}\right)^2},
\end{equation}
where the sum runs over the edges of $\Z^d$.

\medskip
\noindent We proceed in four steps.

\medskip
\step{1} Proof of 
\begin{eqnarray}
\sup_{\omega_e}\left|\frac{\partial \psi_T(0)}{\partial \omega_e}\right|&\lesssim& \big(|\nabla \psi_T(z)|+\mu_d(T)(1+|\nabla \phi_{2T}(z)|)\big)G_T(0,e)\nonumber \\
&& +(1+|\nabla \phi_{2T}(z)|)\int_{\Z^d}G_T(0,w)G_T(e,w)\d w,\label{eq:ap-1}
\end{eqnarray}
where $e=(z,z+\ee_i)$, $G_T(0,e):=G_T(0,z+\ee_i)-G_T(0,z)$, $G_T(e,w)=G_T(z+\ee_i,w)-G_T(z,w)$, and $\mu_d(T)=\left|\begin{array}{rcl}\ln_+T &\text{if}&d=2,\\1 &\text{if}&d>2. \end{array} \right.$
Estimate~\eqref{eq:ap-1} is a direct consequence of \cite[(3.10) \& (3.21)]{Gloria-Otto-09b}, and \cite[(2.14) \& (2.16)]{Gloria-Otto-09}.

\medskip
\step{2} Proof of 
\begin{multline}
\int_{\Z^d} \expec{\big( |\nabla \psi_T(z)|^2+\mu_d(T)^2 (1+|\nabla \phi_{2T}(z)|^2  \big) G_T(0,e)^2}\d x\\
\lesssim 
\left|
\begin{array}{rcl}
T^2\ln^q T &\text{if} & d=2,\\
T &\text{if} & d=3,\\
\ln T &\text{if} &d=4,\\
1 &\text{if} &d>4,
\end{array}
\right.
\label{eq:ap-2}
\end{multline}
where $q$ only depends on the ellipticity constants $\alpha,\beta$.
To prove \eqref{eq:ap-2}, we first replace the gradient of the Green's function by the Green's function itself
and appeal to the deterministic optimal pointwise estimate of \cite[Lemma~4]{Gloria-Otto-09}:
\begin{equation*}
|G_T(x,e)| \,\leq\,G_T(x,z)+G_T(x,z+\ee_i) \,\lesssim\, \mu_d(T) (1+|x-z|)^{2-d} \min\{1,\sqrt{T}|x-z|^{-1}\}.
\end{equation*}
By stationarity, $\expec{|\nabla \psi_T(z)|^2}=\expec{|\nabla \psi_T(0)|^2}$, and $\expec{|\nabla \phi_{2T}(z)|^2} \leq 4\expec{|\phi_{2T}(0)|^2}$,
so that by \cite[Proposition~2.1]{Gloria-Otto-09} and \cite[(3.27) \& (3.29)]{Gloria-Otto-09b},
\begin{eqnarray*}
&& \int_{\Z^d} \expec{\big( |\nabla \psi_T(z)|^2+\mu_d(T)^2 (1+|\nabla \phi_{2T}(z)|^2  \big) G_T(0,e)^2}\d x\\
&\lesssim &(\expec{|\nabla \psi_T(z)|^2}+\mu_d(T)^2\expec{|\phi_{2T}(0)|^2})\int_{\Z^d} \mu_d(T)^2 (1+|z|)^{2(2-d)} \min\{1,\sqrt{T}^2|z|^{-2}\} \d z\\
&\lesssim & \left|
\begin{array}{rcl}
\mu_d(T)^q (T+1) T &\text{if} &d=2,\\
(\sqrt{T}+1)\sqrt{T} &\text{if} &d=3,\\
\ln T+1 &\text{if} &d=4,\\
1&\text{if} &d>4,
\end{array}
\right.
\end{eqnarray*}
which yields \eqref{eq:ap-2} for $T\gg 1$.

\medskip
\step{3} Proof of 
\begin{multline}
\int_{\Z^d} \expec{(1+|\nabla \phi_{2T}(z)|^2)\int_{\Z^d}\int_{\Z^d} G_T(0,w)G_T(0,w')|G_T(e,w)||G_T(e,w')|\d w\d w'}\d z \\
\lesssim 
\left|
\begin{array}{rcl}
T^2\ln^q T &\text{if} &d=2,\\
T^{3-d/2}&\text{if} &5\ge d >2, \\
\ln T &\text{if} &d=6,\\
1 &\text{if} &d>6.
\end{array}
\right.
\label{eq:ap-3}
\end{multline}
where $q$ only depends on the ellipticity constants $\alpha,\beta$.

\noindent We first estimate the Green's function using the deterministic pointwise estimate 
of \cite[Lemma~4]{Gloria-Otto-09}:
\begin{eqnarray}
&& \int_{\Z^d} \expec{(1+|\nabla \phi_{2T}(z)|^2)\int_{\Z^d}\int_{\Z^d} G_T(0,w)G_T(0,w')|G_T(e,w)||G_T(e,w')|\d w\d w'}\d z\nonumber \\
&\lesssim & \int_{\Z^d}\int_{\Z^d} \mu_d(T)^2 (1+|w|)^{2-d}(1+|w'|)^{2-d} \min\{1,\sqrt{T}|w|^{-1}\}^k\min\{1,\sqrt{T}|w'|^{-1}\}^k \nonumber \\
&& \qquad \qquad \qquad \times \int_{\Z^d}  \expec{(1+|\nabla \phi_{2T}(z)|^2)|G_T(e,w)||G_T(e,w')|}\d z\d w\d w' \label{eq:ap-7}
\end{eqnarray}
for some $k\geq 1$ to be fixed later ($k=5$ will be enough).
We then deal with the inner integral, and appeal to the Meyers' estimate of \cite[Lemma~2.9]{Gloria-Otto-09}
and the bounds of \cite[Proposition~2.1]{Gloria-Otto-09} on the moments of the modified correctors. We let
$p>2$ be the Meyers' exponent. By H\"older's inequality in probability with exponents $((p-2)/p,2/p)$, 
Cauchy-Schwarz' inequality, and stationarity of $\nabla G_T$, we have
\begin{eqnarray}
&&\int_{\Z^d}  \expec{(1+|\nabla \phi_{2T}(z)|^2)|G_T(e,w)||G_T(e,w')|}\d z \nonumber\\
&\lesssim & \int_{\Z^d} \expec{1+|\nabla \phi_{2T}(z)|^{2p/(p-2)}}\expec{|G_T(e,w)|^p}^{1/p}\expec{|G_T(e,w')|^p}^{1/p} \d z \nonumber\\
&\leq &\int_{\Z^d} \expec{1+|\nabla \phi_{2T}(z)|^{2p/(p-2)}}\expec{|\nabla_z G_T(z-w,0)|^p}^{1/p}\expec{|\nabla_z G_T(z-w',0)|^p}^{1/p} \d z \nonumber\\
&\lesssim & \mu_d(T)^q \int_{\Z^d} \expec{|\nabla_z G_T(z-w,0)|^p}^{1/p}\expec{|\nabla_z G_T(z-w',0)|^p}^{1/p} \d z.\label{eq:ap-6}
\end{eqnarray}
The combination of \eqref{eq:ap-7} \& \eqref{eq:ap-6} with \cite[Lemma~2.9]{Gloria-Otto-09}  yields
\begin{eqnarray*}
&&\int_{\Z^d} \expec{(1+|\nabla \phi_{2T}(z)|^2)\int_{\Z^d}\int_{\Z^d} G_T(0,w)G_T(0,w')|G_T(e,w)||G_T(e,w')|\d w\d w'}\d z \\
&\lesssim & \mu_d(T)^q \int_{\Z^d}\int_{\Z^d}\int_{\Z^d}g_T(|w|)g_T(|w'|)h_T(z-w)h_T(z-w')\d z\d w\d w',
\end{eqnarray*}
where $g_T(t)=(1+t)^{2-d} \min\{1,\sqrt{T}t^{-1}\}^k$, and $h_T$ is such that: for $R\sim 1$, 
\begin{equation*}
\int_{|x|\leq R}h_T(x)^2\,\lesssim\,1,
\end{equation*}
and for all $R\gg 1$ and all $j\geq 1$,
\begin{equation*}
\int_{2^jR \leq |x|<2^{j+1} R}h_T(x)^2\d x\,\lesssim\, (2^jR)^{d-2(d-1)} \min\{1,\sqrt{T}2^{-j}\}^{k}.
\end{equation*}
As we shall prove in the next step, this implies \eqref{eq:ap-7}.
Combined with Step~1 and Step~2, this proves the lemma.

\medskip
\step{4} Proof of 
\begin{multline}
\int_{\Z^d}\int_{\Z^d}\int_{\Z^d}g_T(|w|)g_T(|w'|)h_T(z-w)h_T(z-w')\d z\d w\d w' \\
\,\lesssim\,
\left|
\begin{array}{rcl}
T^2\ln T&\text{if}&d=2,\\
T^{3-d/2}&\text{if}&5\ge d>2, \\
\ln T &\text{if}&d=6,\\
1&\text{if}&d>6.
\end{array}
\right.
\label{eq:ap-8}
\end{multline}
The proof of \eqref{eq:ap-8} is made technical because the bounds on $h_T$ do hold integrated on dyadic annuli, and not pointwise.
In line with the bounds on $h_T$, we prove the claim by using a doubly dyadic decomposition of $\Z^d\times \Z^d$ 
combined with the results of \cite[Proof of Lemma~2.10, Steps~1,~2 \&~4]{Gloria-Otto-09}, that we
recall for the reader's convenience:
there exists $R\sim 1$ such that for all $i \in \N$,
\begin{eqnarray}\nonumber
\lefteqn{\int_{2^iR<|x|\leq 2^{i+1}R} \int_{|z|\leq |z-x|}h_T(z)h_T(z-x)\d z\d x}
\\
&\lesssim&
\left|
\begin{array}{lcl}
(2^iR)^2\max\{1,\ln(\sqrt{T}(2^iR)^{-1})\}&\text{if}&d=2,\\
(2^iR)^2 &\text{if}&d>2,
\end{array}
\right.  \label{eq:Lemma-2.10-2}
\end{eqnarray}
\begin{equation}
\int_{|x|\leq 4R} \int_{|z|\leq |z-x|}h_T(z)h_T(z-x)\d z\d x \,\lesssim \, \left|
\begin{array}{lcl}
\ln T &\text{if}&d=2,\\
1 &\text{if}&d>2.
\end{array}
\right. \label{eq:Lemma-2.10-1}
\end{equation}
We first use the symmetry with respect to $w$ and $w'$ to restrict the set of integration to $|w'|\geq |w|$, and
we make a change of variables
\begin{eqnarray*}
\lefteqn{\int_{\Z^d}\int_{\Z^d}\int_{\Z^d}g_T(|w|)g_T(|w'|)h_T(z-w)h_T(z-w')\d z\d w'\d w }\nonumber\\
&\leq & 2 \int_{w\in \Z^d}\int_{w'\in \Z^d, |w'|\geq |w|}\int_{\Z^d}g_T(|w|)g_T(|w'|)h_T(z-w)h_T(z-w')\d z\d w'\d w \\
&= &2\int_{w\in \Z^d}\int_{w-w'\in \Z^d,|w'|\geq |w|}\int_{\Z^d}g_T(|w|)g_T(|w'|)h_T(z)h_T(z-(w-w'))\d z\d w'\d w ,
\end{eqnarray*}
followed by the associated doubly dyadic decomposition of space
\begin{eqnarray}
\lefteqn{\int_{\Z^d}\int_{\Z^d}\int_{\Z^d}g_T(|w|)g_T(|w'|)h_T(z-w)h_T(z-w')\d z\d w'\d w }\nonumber\\
&\lesssim & \sum_{i\in \N} \sum_{j\in \N}  \int_{2^i R<|w|\leq 2^{i+1}R}\int_{\tiny \begin{array}{l}2^j R<|w-w'|\leq 2^{j+1}R\\ |w'|\geq |w|\end{array}}g_T(|w|)g_T(|w'|) \nonumber \\
&& \qquad  \qquad  \qquad  \qquad  \qquad  \qquad \times \int_{\Z^d}h_T(z)h_T(z-(w-w'))\d z\d w'\d w \label{eq:ap-9.1}\\
&&+ \sum_{j\in \N}  \int_{|w|\leq R}\int_{\tiny \begin{array}{l}2^j R<|w-w'|\leq 2^{j+1}R\\ |w'|\geq |w|\end{array}}g_T(|w|)g_T(|w'|)  \nonumber \\
&& \qquad  \qquad  \qquad  \qquad  \qquad  \qquad \times \int_{\Z^d}h_T(z)h_T(z-(w-w'))\d z\d w'\d w \label{eq:ap-9.2} \\
&&+ \sum_{i\in \N}  \int_{2^i R<|w|\leq 2^{i+1}R}\int_{\tiny \begin{array}{l}|w-w'|\leq R\\ |w'|\geq |w|\end{array}}g_T(|w|)g_T(|w'|)  \nonumber \\
&& \qquad  \qquad  \qquad  \qquad  \qquad  \qquad \times \int_{\Z^d}h_T(z)h_T(z-(w-w'))\d z\d w'\d w \label{eq:ap-9.3}\\
&&+\int_{|w|\leq R}\int_{\tiny \begin{array}{l}|w-w'|\leq R\\ |w'|\geq |w|\end{array}}g_T(|w|)g_T(|w'|) \nonumber \\
&& \qquad  \qquad  \qquad  \qquad  \qquad  \qquad \times \int_{\Z^d}h_T(z)h_T(z-(w-w'))\d z\d w'\d w, \label{eq:ap-9.4}
\end{eqnarray}
where $R\sim 1$ is as above.
We begin with the last term \eqref{eq:ap-9.4} of the sum, and appeal to \eqref{eq:Lemma-2.10-1} and the definition of $g_T$:
\begin{eqnarray}
&&\int_{|w|\leq R}\int_{\tiny \begin{array}{l}|w-w'|\leq R\\ |w'|\geq |w|\end{array}}g_T(|w|)g_T(|w'|) \int_{\Z^d}h_T(z)h_T(z-(w-w'))\d z\d w'\d w \nonumber \\
&\lesssim & \int_{|w-w'|\leq R}\int_{\Z^d}h_T(z)h_T(z-(w-w'))\d z\d w'\nonumber \\
&\lesssim &\left| 
\begin{array}{rcl}
\ln T &\text{if}&d=2,\\
1 &\text{if}&d>2.
\end{array}
\right. \label{eq:ap-9.4b}
\end{eqnarray}
We continue with \eqref{eq:ap-9.3}. Since $|w-w'|\leq R$, $g_T(|w'|)\sim g_T(|w|)$, and we have using \eqref{eq:Lemma-2.10-1}
and the definition of $g_T$:
\begin{eqnarray}
&&\sum_{i\in \N}  \int_{2^i R<|w|\leq 2^{i+1}R}\int_{\tiny \begin{array}{l}|w-w'|\leq R\\ |w'|\geq |w|\end{array}}g_T(|w|)g_T(|w'|)  
%\nonumber \\
%&& \qquad  \qquad  \qquad  \qquad  \qquad  \qquad \times 
\int_{\Z^d}h_T(z)h_T(z-(w-w'))\d z\d w'\d w \nonumber \\
&\lesssim & \left( \sum_{i\in \N}  \int_{2^i R<|w|\leq 2^{i+1}R}g_T(|w|)^2 \d w\right)\left( \int_{|w-w'|\leq R}\int_{\Z^d}h_T(z)h_T(z-(w-w'))\d z\d w'\right) \nonumber \\
&\lesssim &\left|
\begin{array}{rcl}
T\ln T&\text{if}&d=2,\\
\sqrt{T}&\text{if}&d=3,\\
\ln T &\text{if}&d=4,\\
1 &\text{if}&d>4.
\end{array}
\right. \label{eq:ap-9.3b}
\end{eqnarray}
For \eqref{eq:ap-9.2} we note that $|w|\leq R$ and $2^j R<|w-w'|\leq 2^{j+1}R$ imply that $|w'|\sim 2^j R$, and we
appeal to \eqref{eq:Lemma-2.10-2}:
\begin{eqnarray}
&&\sum_{j\in \N}  \int_{|w|\leq R}\int_{\tiny \begin{array}{l}2^j R<|w-w'|\leq 2^{j+1}R\\ |w'|\geq |w|\end{array}}g_T(|w|)g_T(|w'|) \int_{\Z^d}h_T(z)h_T(z-(w-w'))\d z\d w'\d w \nonumber \\
&\lesssim & g_T(0)R^d\sum_{j\in \N} g_T(2^jR)
%&& \qquad  \qquad  \qquad  \qquad  \qquad  \times
 \int_{2^j R<|w-w'|\leq 2^{j+1}R}\int_{\Z^d}h_T(z)h_T(z-(w-w'))\d z\d w'\nonumber \\
&\lesssim &\left|
\begin{array}{rcl}
T\ln T &\text{if}&d=2,\\
\sqrt{T}&\text{if}&d=3,\\
\ln T &\text{if}&d=4,\\
1 &\text{if}&d>4.
\end{array}
\right. \label{eq:ap-9.2b}
\end{eqnarray}
The dominant term is \eqref{eq:ap-9.1}.
We split the double sum into three parts according to the range of $i$ and $j$:
\begin{itemize}
\item the diagonal part: $|i-j|\leq 1$,
\item the off-diagonal parts: $i\geq j+2$ and $j\geq i+2$.
\end{itemize}
For $|i-j|\leq 1$, we use the inequality $|w|\leq |w'|\leq |w|+|w-w'|$ 
so that for the $i^{th}$ term of the sum, $|w'|\sim 2^iR$.
In particular, using \eqref{eq:Lemma-2.10-2}, this yields for the diagonal term
\begin{eqnarray}
&&\sum_{i\in \N} \sum_{|j-i|\leq 1}  \int_{2^i R<|w|\leq 2^{i+1}R}\int_{\tiny \begin{array}{l}2^j R<|w-w'|\leq 2^{j+1}R\\ |w'|\geq |w|\end{array}}g_T(|w|)g_T(|w'|) \nonumber \\
&& \qquad  \qquad  \qquad  \qquad  \qquad  \qquad \times \int_{\Z^d}h_T(z)h_T(z-(w-w'))\d z\d w'\d w \nonumber \\
&\lesssim & \sum_{i\in \N} (2^iR)^d g_T(2^iR)^2 \int_{2^{i-1} R<|w-w'|\leq 2^{i+2}R}\int_{\Z^d}h_T(z)h_T(z-(w-w'))\d z\d w'\nonumber \\
&\lesssim &\sum_{i\in \N} (2^iR)^{d+2(2-d)} \min\{1,\sqrt{T}(2^iR)^{-1}\}^5   (2^iR)^2\mu_d(T) \nonumber \\
&=&\mu_d(T)\sum_{i\in \N} (2^iR)^{6-d} \min\{1,\sqrt{T}(2^iR)^{-1}\}^5\nonumber \\
&\lesssim &\left|
\begin{array}{rcl}
{T}^2\ln T &\text{if}&d=2,\\
\sqrt{T}^3 &\text{if}&d=3,\\
{T}&\text{if}&d=4,\\
\sqrt{T}&\text{if}&d=5,\\
\ln T &\text{if}&d=6,\\
1 &\text{if}&d>6.
\end{array}
\right. \label{eq:ap-9.1-1}
\end{eqnarray}
We turn to the first off-diagonal term: those integers $i,j$ such that $i\geq j+2$.
In this case, we use the estimate $|w-w'|-|w|\leq |w'|\leq |w-w'|+|w|$, which shows that 
for the $(i,j)^{th}$ term of the sum, $|w'|\sim 2^iR$. In particular, using \eqref{eq:Lemma-2.10-2}, this yields for the
first off-diagonal term
\begin{eqnarray}
&&\sum_{i\in \N} \sum_{j\leq i-2}  \int_{2^i R<|w|\leq 2^{i+1}R}\int_{\tiny \begin{array}{l}2^j R<|w-w'|\leq 2^{j+1}R\\ |w'|\geq |w|\end{array}}g_T(|w|)g_T(|w'|) \nonumber \\
&& \qquad  \qquad  \qquad  \qquad  \qquad  \qquad \times \int_{\Z^d}h_T(z)h_T(z-(w-w'))\d z\d w'\d w \nonumber \\
&\lesssim & \sum_{i\in \N} (2^iR)^d g_T(2^iR)^2 \sum_{j\leq i-2}\int_{2^j R<|w-w'|\leq 2^{j+1}R}\int_{\Z^d}h_T(z)h_T(z-(w-w'))\d z\d w'\nonumber \\
&\lesssim &\sum_{i\in \N} (2^iR)^{d+2(2-d)} \min\{1,\sqrt{T}|2^iR|^{-1}\}^5   \sum_{j\leq i-2}(2^jR)^2\mu_d(T)\nonumber\\
&\lesssim &\sum_{i\in \N} (2^iR)^{d+2(2-d)} \min\{1,\sqrt{T}|2^iR|^{-1}\}^5   (2^iR)^2\mu_d(T) \nonumber\\
&\lesssim &\left|
\begin{array}{rcl}
{T}^2\ln T &\text{if}&d=2,\\
\sqrt{T}^3 &\text{if}&d=3,\\
{T}&\text{if}&d=4,\\
\sqrt{T}&\text{if}&d=5,\\
\ln T &\text{if}&d=6,\\
1 &\text{if}&d>6.
\end{array}
\right. \label{eq:ap-9.1-2}
\end{eqnarray}
We now treat the last term of the sum, that is those integers $i,j$ such that $j\geq i+2$.
Then, similarly to \eqref{eq:ap-9.1-2} we deduce that for $(i,j)^{th}$ term of the sum, $|w'|\sim 2^jR$. 
Hence, using \eqref{eq:Lemma-2.10-2}, we obtain
\begin{eqnarray}
&&\sum_{i\in \N} \sum_{j\geq i+2}  \int_{2^i R<|w|\leq 2^{i+1}R}\int_{\tiny \begin{array}{l}2^j R<|w-w'|\leq 2^{j+1}R\\ |w'|\geq |w|\end{array}}g_T(|w|)g_T(|w'|) \nonumber \\
&& \qquad  \qquad  \qquad  \qquad  \qquad  \qquad \times \int_{\Z^d}h_T(z)h_T(z-(w-w'))\d z\d w'\d w \nonumber \\
&\lesssim & \sum_{i\in \N} (2^iR)^d g_T(2^iR)\sum_{j\geq i+2}g_T(2^jR) \int_{2^j R<|w-w'|\leq 2^{j+1}R}\int_{\Z^d}h_T(z)h_T(z-(w-w'))\d z\d w'\nonumber \\
&=& \sum_{j\in \N} \left(g_T(2^jR) \int_{2^j R<|w-w'|\leq 2^{j+1}R}\int_{\Z^d}h_T(z)h_T(z-(w-w'))\d z\d w'\right)\sum_{i\leq j-2}(2^iR)^d g_T(2^iR) \nonumber \\
&\lesssim & \sum_{j\in \N} g_T(2^jR) (2^j R)^2\mu_d(T) \sum_{i\leq j-2}(2^iR)^d g_T(2^iR) \nonumber \\
&\lesssim &\left|
\begin{array}{rcl}
{T}^2\ln T &\text{if}&d=2,\\
\sqrt{T}^3 &\text{if}&d=3,\\
{T}&\text{if}&d=4,\\
\sqrt{T}&\text{if}&d=5,\\
\ln T &\text{if}&d=6,\\
1 &\text{if}&d>6.
\end{array}
\right. \label{eq:ap-9.1-3}
\end{eqnarray}
as for the first off-diagonal term.

\medskip
\noindent Estimate \eqref{eq:ap-8} then follows from the combination of \eqref{eq:ap-9.1}--\eqref{eq:ap-9.4} 
with \eqref{eq:ap-9.1-1}--\eqref{eq:ap-9.1-3} and \eqref{eq:ap-9.4b}--\eqref{eq:ap-9.2b}. 

\section{Numerical tests in the discrete periodic case}\label{append-perio}

Numerical tests of \cite{Gloria-09} have confirmed the sharpness of Theorem~\ref{th:error-estim-per} 
for the approximation $A_{\mu,1,R,L}$ on a discrete periodic example.
In the present work, we consider the same discrete example, and numerically check the asymptotic
convergence of $A_{\mu,2,R,L}$ to $A_\ho$.
As expected, the systematic error is reduced, and the limiting factor rapidly becomes the machine precision.
%
% check numerically the interest of Theorem~\ref{th:error-estim-per},
% we turn to the examples studied in \cite{Gloria-09}.
% In particular, we treat a discrete example --- which enables us to consider a large number of periodic cells, 
% as well as a continuous example --- for which the number of periodic cells does not exceed 50.
%
% \medskip
%
% \noindent 
% We start with the discrete case.
% 
The discrete corrector equation we consider is
\begin{equation*}
-\nabla^*\cdot A(\xi+\nabla \phi)=0 \qquad \text{ in }\Z^2,
\end{equation*}
where $\nabla$ and $\nabla^*$ are as in \eqref{eq:disc-nabla}, and 
\begin{equation*}
A(x):=\dig{\omega_{x,x+\ee_1},\omega_{x,x+\ee_2}}.
\end{equation*}
The matrix $A$ is $[0,4)^2$-periodic, and sketched on a periodic cell on Figure~\ref{fig:discret}. 
In the example considered, $\omega_{x,x+\ee_1}$ and $\omega_{x,x+\ee_2}$ represent 
the conductivities $1$ or $100$ of the horizontal edge $(x,x+\ee_1)$ and the vertical edge $(x,x+\ee_2)$ respectively,
according to the colors on Figure~\ref{fig:discret}.
The homogenization theory for such discrete elliptic operators is similar to the continuous case (see for instance \cite{Vogelius-91} 
in the two-dimensional case dealt with here). 
By symmetry arguments, the homogenized matrix associated with $A$ is a multiple of the identity.
It can be evaluated numerically (note that we do not make any other error than the machine precision). Its numerical value
is $A_\ho=26.240099009901\dots$.
\begin{figure}
\centering
\includegraphics[scale=.35]{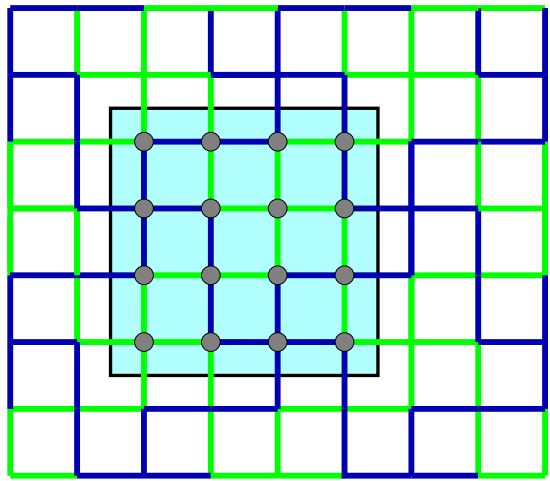}
\caption{Periodic cell in the discrete case}
\label{fig:discret}
\end{figure}
To illustrate Theorem~\ref{th:error-estim-per} in its discrete version (which is similar, see
\cite{Gloria-09} for related arguments), we have conducted a series of tests for $A_{\mu,2}$.
In particular, we have taken $\mu\sim R^{-3/2}$, $L=R/3$,
and a filter of infinite order.
In this case, the convergence rate is expected to be of order $3$ for $A_{\mu,1}$, and of order $6$ for $A_{\mu,2}$.
This is indeed the case, as can be seen on Figure~\ref{fig2}, where $R$ denotes the number of periodic cells and ranges from $6$ to
$400$ (that is $\log(R)$ up to $2.6$).
\begin{figure}
\centering
\psfrag{hh}{$\log R$}
\psfrag{kk}{$\log \text{Error}(\mu,R)$}
\includegraphics[scale=.45]{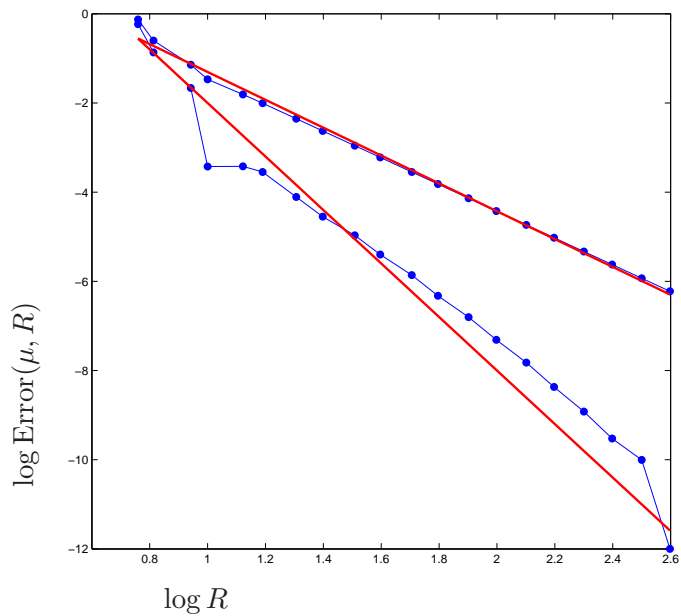}
\caption{Absolute error in log scale for $\mu=250 \,R^{-3/2}$, $A_{\mu,1,R,L}$ (slope $-3.1$) and $A_{\mu,2,R,L}$ (slope $-6$), filter of infinite order.}
\label{fig2}
\end{figure}

\end{document}